\documentclass{amsart}
\usepackage{amssymb}
\usepackage{amsmath}
\usepackage{amsfonts,latexsym,amstext}

\newtheorem{theorem}{Theorem}[section]
\newtheorem{definition}[theorem]{Definition}

\newtheorem{lemma}[theorem]{Lemma}
\newtheorem{proposition}[theorem]{Proposition}
\newtheorem{corollary}[theorem]{Corollary}

\newcommand{\hdim}{\dim_{\mathsf{H}}}

\newcommand{\R}{{\mathbb R}}
\newcommand{\C}{{\mathcal C}}
\newcommand{\Z}{{\mathbb{Z}}}

\newcommand{\N}{{\mathbb N}}
\newcommand{\X}{\mathbb{X}}
\newcommand{\W}{\mathbb{W}}
\newcommand{\Y}{\mathbb{Y}}

\newcommand{\F}{{\mathcal F}}

\newcommand{\bt}{\mathbf{T}}

\newcommand{\om}{\omega}
\newcommand{\pom}{{(\omega)}}

\def\R{\mathbb{R}}
\def\N{\mathbb{N}}

\def\F{\mathcal F}

\newcommand{\eps}{\varepsilon}
\newcommand{\oeta}{\overline{\eta}}
\newcommand{\otheta}{\overline{\vartheta}}
\newcommand{\opi}{\overline{\Pi}}
\newcommand{\supp}{\text{supp}}

\begin{document}

\title{$L^q$ dimensions and projections of random measures}

\author{Daniel Galicer}

\address{Departamento de Matem\'{a}tica and IMAS/CONICET\\
Facultad de Cs. Exactas y
Naturales\\
Universidad de Buenos Aires \\
Buenos Aires, Ciudad Universitaria, Pab. I (1428), Argentina.}

\email{dgalicer@dm.uba.ar}

\author{Santiago Saglietti}

\address{Departamento de Matem\'{a}tica and IMAS/CONICET\\
Facultad de Cs. Exactas y
Naturales\\
Universidad de Buenos Aires \\
Buenos Aires, Ciudad Universitaria, Pab. I (1428), Argentina.
}

\email{ssaglie@dm.uba.ar}

\author{Pablo Shmerkin}

\address{Departamento de Matem\'{a}ticas y Estad\'{\i}sticas and CONICET\\
Universidad Torcuato Di Tella\\
Av. Figueroa Alcorta 7350 (C1428BCW), Buenos Aires, Argentina.}

\urladdr{http://www.utdt.edu/profesores/pshmerkin}
\email{pshmerkin@utdt.edu}

\author{Alexia Yavicoli}

\address{Departamento de Matem\'{a}tica and IMAS/CONICET\\
Facultad de Cs. Exactas y
Naturales\\
Universidad de Buenos Aires \\
Buenos Aires, Ciudad Universitaria, Pab. I (1428)\\ Argentina.}

\email{ayavicoli@dm.uba.ar}

\thanks{D.G. was supported by projects CONICET PIP 0624, PICT 2011-1456 and
UBACyT 20020130300057BA. S.S. was supported by projects PICT 2012-2744 and UBACyT 200120120100151GC. P.S. was supported by projects PICT 2011-4036 and PICT 2013-1393 (ANPCyT). A.Y.  was supported by projects UBACyT 2014-2017 20020130100403BA and PIP 11220110101018 (CONICET)}


\begin{abstract}
 We prove preservation of $L^q$ dimensions (for $1<q\le 2$) under all orthogonal projections for a class of random measures on the plane, which includes (deterministic) homogeneous self-similar measures and a well-known family of measures supported on $1$-variable fractals as special cases. We prove a similar result for certain convolutions, extending a result of Nazarov, Peres and Shmerkin. Recently many related results have been obtained for Hausdorff dimension, but much less is known for $L^q$ dimensions.
\end{abstract}

\maketitle

\section{Introduction}

In recent years there has been great interest in understanding the size of linear (and non-linear) images of sets and measures of dynamical and arithmetic origin. Here ``size'' may refer to some fractal dimension, or to Lebesgue measure/absolute continuity.

Even if one is concerned with sets, the proofs usually involve measures supported on them; in this article, we deal primarily with measures (by a measure we always mean a Borel locally finite measure on some Euclidean space). If $\mu$ is a measure on a space $X$ and $f:X\to Y$ is a map, we denote the push-forward of $\mu$ via $f$ by $f\mu$, that is, $f\mu(B)=\mu(f^{-1}B)$ whenever $f^{-1}(B)$ is measurable. A guiding heuristic principle is that if $\mu$ is a measure on $\R^d$, $\Pi:\R^d\to\R^k$ is a ``nice'' Lipschitz map and $\dim$ is some notion of dimension for measures, then ``typically'' $\Pi\mu$ is ``as large as possible'' in the sense that $\dim\Pi\mu=\dim\mu$ if $\dim\mu\le k$, and $\dim\Pi\mu=k$ if $\dim\mu>k$ (in the latter case, one expects $\Pi\mu$ to also be absolutely continuous).

A precise version of this heuristic is given by Marstrand's projection theorem (and its variants) which, for the case of measures, says that, for any measure $\mu$ on $\R^d$, there is an equality $\dim\Pi\mu = \min(\dim\mu,k)$ for almost all linear maps $\Pi:\R^d\to\R^k$, whenever $\dim$ is either Hausdorff or $L^q$ dimension ($1< q\le 2$); these notions of dimension will be defined later. See e.g. \cite[Chapter 9]{Mattila95} and \cite[Theorem 1.1]{HuntKaloshin97}. However, for measures with a dynamical or arithmetic structure, such as self-similar measures or measures invariant under some algebraic dynamical system, one would like to say more, ideally finding the precise set of exceptional linear maps $\Pi$.

Early results of this type for sets were obtained in \cite{Moreira98, PeresShmerkin09, FJS10}. Recall that the (lower) Hausdorff dimension of a measure $\mu$ is
\[
\hdim\mu = \inf\{ \hdim A: \mu(A)>0\},
\]
where $\hdim A$ is the Hausdorff dimension of $A$.  A general method to bound the Hausdorff dimension of projected measures was developed in \cite{HochmanShmerkin12}, with variants and applications given in \cite{FalconerJin14, Almarza14, Farkas14, FFS15}. Among other things, the equality $\hdim\Pi\mu=\min(\hdim \mu,k)$ is established for many classes of measures (satisfying certain necessary assumptions), including self-similar measures, more general random cascades on self-similar sets,  products of $\times m$-invariant measures on $[0,1]$, and Bernoulli and Gibbs measures for the natural symbolic coding of the $(\times m,\times n)$-toral automorphisms, and \emph{all} linear maps $\Pi$ (apart from obvious exceptions). A recent breakthrough on the dimensions of self-similar measures \cite{Hochman14} also has applications on the dimension of projections, see \cite{ShmerkinSolomyak15}; this work again deals with Hausdorff dimensions of measures.

Although Hausdorff dimension is no doubt highly relevant, there are many other concepts of dimension of a measure which are also important both mathematically and in applications. Chief among them is the one-dimensional parameter family of dimensions known as $L^q$ dimensions: let
\begin{equation} \label{eq:def-Cq}
\mathcal{C}^q_{\mu} (n) := \sum_{I\in\mathfrak{D}_n} \left(\mu(I)\right)^q,
\end{equation}
where $\mathfrak{D}_n$ is the family of dyadic cubes $\{2^{-n} \cdot ([0,1)^d + j) : j \in \Z^d\}$ in $\R^d$. For $q>0, q\neq 1$, the lower $L^q$ dimension $\underline{D}_q(\mu)$ is the (lower) suitably normalized scaling exponent of $\mathcal{C}^q_\mu(n)$ as $n\to\infty$:
\[
\underline{D}_q(\mu) = \liminf_{n\to\infty} \frac{\log \mathcal{C}^q_{\mu} (n)}{-n(q-1)}.
\]
For simplicity we always take logarithms to base $2$, unless otherwise noted. The upper $L^q$ dimension $\overline{D}_q$ is defined analogously. When the limit in question exists, it is denoted $D_q(\mu)$; in this case we say that the $L^q$ dimension exists. This family of dimensions measures the degree of singularity of a measure according to its global fluctuations, and are a central ingredient of the multifractal formalism. Of special relevance is the value $q=2$; $D_2$ is also known as the \emph{correlation dimension} of $\mu$. This is partly because (lower) correlation dimension can also be defined in terms of energies:
\[
\underline{D}_2\mu = \sup\left\{ s\ge 0  : \int\int |x-y|^{-s} \,d\mu(x)\,d\mu(y) <\infty\right\}.
\]

The map $q\mapsto \underline{D}_q\mu$ is non-increasing, and $\underline{D}_q\mu \le \hdim\mu\le \overline{D}_{q'}$ for $q'<1<q$ (see e.g. \cite{FLR02}). In general, $q\mapsto D_q$ may be strictly decreasing (this is a reflection of the multifractality of $\mu$), but it may also be constant. For example, if $\mu$ is Ahlfors-regular with exponent $d$ (that is, if $C^{-1}\, r^d \le \mu(B(x,r))\le C\, r^d$ for all $x\in\supp(\mu)$) then $D_q\mu=\hdim\mu=d$ for all $q$. For many measures of dynamical origin, such as self-similar measures, the limit in the definition of $D_q$ is known to exist, see \cite{PeresSolomyak00}.

The only previous result on $L^q$ dimensions of projected measures was obtained in \cite{NPS12}. There it is proved that if $\mu,\nu$ are self-similar measures satisfying certain natural assumptions, then for any $q\in (1,2]$,
\[
D_q(\Pi(\mu\times\nu)) = \min(D_q(\mu\times\nu),1)
\]
for all orthogonal projections $\Pi$ onto lines, other than the principal ones (which are clearly exceptional for products).

In this article we prove preservation of $L^q$ dimensions for $q\in (1,2]$ under all projections, for a class of planar measures which include certain self-similar and stochastically self-similar measures, and for certain products of two measures. Precise definitions are given in the next section. Among other applications, we improve upon the main result of \cite{NPS12} in several different directions, and obtain a different (and somewhat more elementary) proof of a projection result from \cite{HochmanShmerkin12} and sharpen it in some special cases.

We follow the general approach of \cite{NPS12}, with suitable variants. A central element in the main result of \cite{NPS12} is the existence of certain subadditive cocycle over an irrational rotation. In the present setting, there is also a subadditive cocycle at the core of the proofs, but the base transformation is now a circle extension of a shift space. Most of the additional work is then concerned with studying this somewhat more complex dynamical object. Nevertheless, we also introduce some generalizations and clarifications that are valid also in the deterministic setting of
\cite{NPS12}.

\section{Main results}

\subsection{The model}
\label{subsec:themodel}

Our general setup is as follows. A \emph{rule} is an iterated function system $(f_1,\ldots,f_k)$, where each map $f_j$ is a strictly contractive similarity on $\R^d$ (the ambient dimension $d$ will always be either $1$ or $2$, later on we will impose an additional homogeneity assumption on the rules). We will work with a finite set of $N$ rules $(f_1^{(i)}, \dots, f_{k_i}^{(i)})$, $i\in\{1,\ldots,N\}$. Since the maps $f_j^{(i)}$ are uniformly contractive, if $R>0$ is sufficiently large, then $f_j^{(i)}(B[0,R])\subset B[0,R]$ for all $i\in\{1,\ldots,N\}$, $j\in\{1,\ldots, k_i\}$, where $B[0,R]$ stands for the closed ball of radius $R$ centered at the origin.

Given a sequence $\omega=(\omega_n)_{n \in \N} \in Y:=\{1,\dots,N\}^{\N}$ we define the space of words of length $n$ (possibly with $n=\infty$) with respect to $\omega$ by the formula
\[
\mathbb{X}^{(\omega)}_n:=  \prod_{j=1}^n \{1,\dots,k_{\omega_j}\}.
\]
Note that all $\mathbb{X}^{(\omega)}_n$ are subsets of a common tree $\X_n:=\prod_{j=1}^n \{1,\ldots,k_{\max}\}$, where $k_{\max}=\max_{i=1}^N k_i$.

For each $n \in \N$ and $u \in \X_n^\pom$ we consider the ball
\[
B^{(\omega)}_{u}=f^{(\omega)}_u(B[0,R]),
\]
where  $f^{(\omega)}_u=f^{(\omega_1)}_{u_1} \circ \dots \circ f^{(\omega_n)}_{u_n}$. We define a compact set
\[
\mathcal{C}^{(\omega)}:= \bigcap_{n \in \N} \bigcup_{u \in \mathbb{X}^{(\omega)}_n} B^{(\omega)}_{u}.
\]
Note that, for every $n$, we have the inclusion $B^{(\omega)}_{ul} \subset B^{(\omega)}_{u}$, for each $u \in \mathbb{X}^{(\omega)}_n$ and  $l \in \{1,\dots,k_{\omega_{n+1}} \}$ (where $ul$ denotes the concatenation of $u$ and $l$). In other words, these disks are nested. Moreover, their diameters tend to zero uniformly. Alternatively, $\C^{\pom}=\Delta_{\omega}\left(\mathbb{X}^{(\omega)}_\infty\right)$, where $\Delta_\omega$ is the \emph{coding map} given by
\[
 \{\Delta_\om(u)\} = \bigcap_{n=1}^\infty B^\pom_{u|n},
\]
where $u|n$ is the restriction of the infinite word $u$ to its first $n$ coordinates. Given $u\in\X_n^\pom$, we also define the \emph{cylinder} $[u]_\om$ as the set of infinite words in $\X_\infty^\pom$ that start with $u$, and note that $\Delta_\om([u]_\om)\subset B_u^\pom$.

We remark that we do not assume that $\{ B^{(\omega)}_{u}:u \in \mathbb{X}^{(\omega)}_n\}$ are disjoint or any other separation condition. Moreover, we do not exclude the possibility that there is a single rule $(N=1)$, in which case $\mathcal{C}^{(\omega)}$ is a deterministic self-similar set.

Even though $\mathcal{C}^{(\omega)}$ is defined for every $\omega$, our results will be probabilistic in nature, and we will be drawing $\omega$ according to an invariant ergodic probability measure $\mu$ for the left-shift $\bt$ on $Y$.

Similar models have been considered in the literature, sometimes under the names ``homogeneous random fractal'' or $1$-variable fractal, see e.g. \cite{hambly92, stenflo01, BHS08}.

We will not be interested in the sets $\mathcal{C}^{\pom}$ themselves, but rather in measures supported on them. For each $i$, let $p_i=(p_1^{(i)},\ldots,p_{k_i}^{(i)})$ be a probability vector. On each $\X_\infty^\pom$ we can then define the product measure
\[
 \oeta^\pom =  \prod_{n=1}^\infty p_{\omega_i}.
\]
The projection of $\oeta^\pom$ via the coding map is a Borel probability measure  $\eta^\pom$ on $\C^\pom$. In the deterministic case $N=1$, this is simply a self-similar measure on $\C^\pom$. The random case arises naturally, even if a priori one is interested only in deterministic self-similar measures. For example, conditional measures on slices of (deterministic) self-similar and self-affine measures often have this form, and one can decompose an arbitrary self-similar measure as $\int \eta^\pom\, d\mu(\om)$ for an appropriate choice of weights $p_i$ and measure $\mu$; see Section \ref{subsec:hochman} below.

Although the family of product measures just described provides our main class of examples, the proofs extend to more general families $\{ \oeta^\pom: \om\in Y\}$  of Borel probability measures on $\mathbb{X}_\infty$ satisfying the following conditions:
\begin{enumerate}
\item[(a)] Each $\oeta^{(\omega)}$ is supported on $\X_\infty^\pom$.
\item[(b)]  The map $\omega \mapsto \oeta^{(\omega)}$ is continuous (considering the weak topology on the space of Borel probability measures on $\X_\infty$, where $\X_\infty$ is endowed with the product topology). In other words, for any continuous function $g$ on $\X_\infty$, we have
$$
\lim_{\omega' \rightarrow \omega} \int g(u) d\oeta^{(\omega')}(u) = \int g(u) d\oeta^{(\omega)}(u).
$$
\item[(c)]  There exists $K \geq 1$ such that for every $\omega \in Y$ the measure $\oeta^{(\omega)}$ satisfies
\begin{equation}\label{eq:measure-sub-product}
\oeta^{(\omega)}\left([uv]_\om\right)  \leq K \,\oeta^{(\omega)}([u]_\om) \, \oeta^{\left(\mathbf{T}^n\omega\right)}([v]_{\mathbf{T}^n\om}),
\end{equation}
for any $u \in  \mathbb{X}^{(\omega)}_n$, $v\in \X^{(\mathbf{T}^n \om)}_m$.
\end{enumerate}

When $\oeta^\pom$ is a product measure as above, this condition holds with equality and $K=1$. This suggests that, in general, the measures $\oeta^\pom$ satisfying (2) (or rather their projections $\eta^\pom$ under $\Delta_\om$) can be thought of as satisfying some kind of ``sub-self-similarity''. In the deterministic case (in which there is a unique measure $\oeta$), Gibbs measures for H\"{o}lder potentials also satisfy \eqref{eq:measure-sub-product}.
We also note that, since cylinders generate the Borel $\sigma$-algebra of each space $\X_\infty^\pom$, it follows from \eqref{eq:measure-sub-product} that
\begin{equation} \label{eq:measure-submultiplicative}
 \oeta^{(\omega)}\left(\{uy:y\in A\}\right)  \leq K\, \oeta^{(\omega)}([u]_\om) \, \oeta^{\left(\mathbf{T}^n\omega\right)}(A),
\end{equation}
for any $u \in  \mathbb{X}^{(\omega)}_n$ and Borel set $A\subset\X^{(\mathbf{T}^n\om)}_\infty$.

\subsection{$L^q$ dimensions of projections}

Now we specialize to $d=2$ and assume that the rules have the form $\{f_1^{(i)}, \dots, f_{k_i}^{(i)}\}$, where  the function $f_j^{(i)}:\mathbb{R}^2 \to \mathbb{R}^2$ is a similarity defined by
\begin{equation} \label{eq:form-rules}
f_j^{(i)}(x):=\lambda_i R_{\alpha_i} x + t_{j}^{i},
\end{equation}
where $\lambda_i \in (0,1)$, $t_{j}^{i} \in \mathbb{R}^2$ and $R_{\alpha_i}$ is the rotation matrix of angle $\alpha_i\in [0,2\pi)$. In other words, each rule is a \emph{homogeneous} IFS (only the translations differ).

We consider the unit circle $S^1$ endowed with the corresponding normalized Haar measure $\mathcal{L}$. Furthermore, define the continuous map $\alpha : Y \rightarrow S^1$ by the formula $\alpha(\omega):=e^{-i\alpha_{\omega_1}}$, and the skew-product map $\mathbf{S}$ on $Y\times S^1$ as
\begin{equation} \label{eq:def-skew-product}
\mathbf{S}(\omega,v) = \Big(\mathbf{T}(\omega), \alpha(\om)v  \Big).
\end{equation}

Recall that $\om$ is said to be $\mu$-\textit{generic} if $\frac1n\sum_{i=1}^n \delta_{\bt^i \om}$ converges to $\mu$ (here, and throghout the paper, convergence of probability measures is understood to be weak convergence). The orthogonal projection onto the line generated by $v\in S^1$ (identified with $\R$) is denoted $\Pi_v$, i.e. $\Pi_v(x,y)= \langle (x,y), v\rangle$.

We can now state our first main theorem.
\begin{theorem} \label{thm:main-projections}
Suppose $N$ rules of the form \eqref{eq:form-rules} are given. Let $\mu$ be a $\bt$-invariant and ergodic measure on $\{1,\ldots,N\}^\N$.  Let $\{ \oeta^\pom\}_{\om\in\{1,\ldots,N\}^\N}$ be a family of measures satisfying (a)-(c) above, and write $\eta^\pom=\Delta_\om\oeta^\pom$ for the projection of $\oeta$ onto $\C^\pom$. Assume furthermore that the product measure $\mu\times\mathcal{L}$ is ergodic for the skew-product $\mathbf{S}$ defined in \eqref{eq:def-skew-product}.

Then for each $q\in (1,2]$ there is a number $D(q)$, such that for $\mu$-almost all $\om$ it holds that $D_q(\eta^\pom)=D(q)$, and
\begin{equation} \label{eq:main-result}
D_q(\Pi_v\eta^\pom) = \min(D(q),1) \quad\text{for all }v\in S^1.
\end{equation}
Furthermore, the convergence of $-\tfrac{\log \mathcal{C}_{\Pi_v\eta^\pom}^q(n)}{n(q-1)}$ to $D_q(\Pi_v\eta^\pom)$ is uniform in $v\in S^1$.

If $D_q(\eta^\pom)=D(q)$ for all $\mu$-generic points $\om$, or if $D(q)\ge 1$, then the above conclusions hold for all $\mu$-generic $\om$.
\end{theorem}

Examples and applications of this result will be discussed in Section \ref{sec:examples}. The assumption that each rule is homogeneous is critical for our method, and it would be interesting to know whether it can be dropped (we recall that for Hausdorff dimension there are similar results which do not require homogeneity, see \cite{HochmanShmerkin12, FalconerJin14}).

\subsection{Convolutions of Cantor measures}

Recall that the \emph{convolution} $\mu*\nu$ of two measures $\mu,\nu$ on $\R^d$ is the push-down of the product $\mu\times\nu$ under the addition map $(x,y)\mapsto x+y$. In this section we will be concerned with convolutions of two measures on $\R$, one of which is a deterministic measure supported on a self-similar set, while the other is a random measure satisfying properties analogous to those of the previous section.

Fix $N$ rules of the form $\{ f_1^{(i)},\ldots, f_{k_i}^{(i)} \}$, where $f_j^{(i)}(x)=a_i x+t_j^i$ for some ratios $a_i\in (0,1)$ and translations $t_j^i\in\R$, and let $\{\overline{\nu}^\pom: \om\in Y\}$ be a family of measures satisfying assumptions (a)-(c) above (with $\overline{\nu}^\pom$ in place of $\oeta^\pom$). We will assume without loss of generality that $f_j^{(i)}([0,1])\subset [0,1]$; we can always achieve this via an affine change of coordinates, which will not affect the statement of the theorem. As before, we denote $\nu^\pom=\Delta_\om\overline{\nu}^\pom$, where $\Delta_\om$ is the coding map. The measures $\nu^\pom$ are then supported on the Cantor sets $\mathcal{C}^\pom\subset [0,1]$ constructed from the sequence of rules $\om$.

We consider yet another rule $\{ g_1,\ldots, g_{k'}\}$, where $g_j(x)=b x+ t'_j$ for some contraction $b\in (0,1)$ and translations $t'_j\in \R$. Again, we assume that $g_j([0,1])\subset [0,1]$ for all $j\in\{1,\ldots,k'\}$. This is a special case of the preceding framework with $N=1$ rule, but we repeat some definitions for the sake of fixing notation. We denote the code space by $\X'_n=\{1,\ldots,k'\}^n$ (again allowing $n=\infty$), and the coding map by $\Delta':\X'_\infty\to [0,1]$, that is,
\[
\{\Delta'(u)\} = \bigcap_{n=1}^\infty g_{u_1}\circ\cdots\circ g_{u_n}([0,1]).
\]
The cylinder of infinite words in $\X'_\infty$ starting with $u$ will be denoted simply by $[u]$.

Further, let $\overline{\vartheta}$ be a Borel probability measure satisfying the analog of (c) above in the random case, that is, we assume that
\begin{equation} \label{eq:sub-theta}
\otheta([uv]) \le K' \, \otheta[u]\, \otheta[v],
\end{equation}
for some constant $K'>0$, and set $\vartheta=\Delta'\otheta$.

We fix $r \in \N$ such that
\begin{equation}\label{control1}
1 < \min_{i=1,\dots,N} \frac{b}{a_i^r}.
\end{equation} and we also consider $l \in \N$ satisfying
\begin{equation}\label{control}
\max_{i=1,\dots,N}{\frac{b}{a_i^r}} < b^{-l}.
\end{equation} Write $\beta:=\ln(b^{-l})$ and $\alpha_{i_1 \dots i_r} = \ln\left(\frac{b}{a_{i_1}\dots a_{i_r}}\right)$ for each choice of indices $1 \leq i_1,\dots,i_r \leq N$. (Note these are natural logarithms.)

Notice that $0 < \alpha_{i_1 \dots i_r} < \beta$ for all $1 \leq i_1,\dots,i_r \leq N$ and also that $\beta$ can be made arbitrarily large by choosing $l$ appropriately. Consider the space $S^1_\beta$ obtained by taking the interval $[-\beta,\beta)$ and identifying its ends, i.e. $-\beta=\beta$, and endow it with the normalized Lebesgue measure $\mathcal{L}_\beta$. Furthermore, define the continuous map $\alpha : Y \rightarrow \R$ by the formula $\alpha(\omega):=\alpha_{\omega_1 \dots \omega_r}$ and the skew-product map $\mathbf{S}$ on $Y\times S^1_\beta$ as
\[
\mathbf{S}(\omega,s) = (\mathbf{T}^r(\omega), s +_{(\beta)} \alpha(\omega)),
\]
where, as before, $\mathbf{T}$ denotes the left shift operator on $Y$, and $+_{(\beta)}$ stands for the natural sum in $S^1_\beta$.
Also, for each $n \in \N$ let us define the $n$-th rotation $\mathbf{R}^n : Y \times S^1_\beta \rightarrow S^1_\beta$ by the formula
\[
\mathbf{R}^n (\omega,s) := \pi_{S^1_\beta} \left( \mathbf{S}^n (\omega,s) \right) = s +_{(\beta)} \alpha_{\omega_1\dots \omega_r} +_{(\beta)} \cdots +_{(\beta)} \alpha_{\omega_{r(n-1)+1} \dots \omega_{rn}},
\]
where $\pi_{S^1_\beta}$ denotes the projection from $Y \times S^1_\beta$ onto $S^1_\beta$.

We can now state our main result on convolutions.

\begin{theorem}\label{thm:mainconvolutions}
Let $\mu$ be an ergodic, invariant measure for $(\{1,\dots,N\}^\N ,\bt)$ such that the product measure $\mu \times \mathcal{L}_{\beta}$ is ergodic for the dynamical system $(\{1,\dots,N\}^\N \times S^1_{\beta}, \mathbf{S})$.

Then for each $q\in (1,2]$ there is a number $D(q)$, such that for $\mu$-almost all $\om$ it holds that $D_q(\nu^\pom\times\vartheta)=D(q)$, and
\begin{equation} \label{eq:main-convolutions}
D_q(\nu^\pom * A_t\vartheta) = \min(D(q),1) \quad\text{for all }t\in [e^{-\beta},e^{\beta}),
\end{equation}
where $A_t(x)=tx$ scales by $t$. Furthermore, the convergence of $-\tfrac{\log \mathcal{C}_{\nu^\pom * A_s\vartheta}^q(n)}{n(q-1)}$ to $D_q(\nu^\pom * A_s\vartheta)$ is uniform in $t\in [e^{-\beta},e^{\beta})$.

If $D_q(\nu^\pom)=D(q)$ for all $\mu$-generic points $\om$, or if $D(q)\ge 1$, then the above conclusions hold for all $\mu$-generic $\om$.
\end{theorem}
 Note that $(x,y)\mapsto x+A_t y$ is, up to affine homeomorphism, the orthogonal projection with angle $\arctan(t)$; hence this result can also be interpreted in terms of projctions of the prouct measure $\nu^\pom\times\vartheta$.  Again, the homogeneity assumption on the rules is crucial. Also, we do not know if the statement holds if $\vartheta$ is also chosen randomly according to a sequence of rules. Although this appears natural, it does not seem possible to build a cocycle like the one at the core of the proof in this setting.

\section{Some auxiliary results from ergodic theory}\label{ergodic}

In this section we collect some ergodic-theoretic facts. The results are rather standard, but we include proofs for completeness, as we have not been able to find exact references. Since the proofs are the same, we state them in greater generality than needed in our later applications.

\subsection{Compact group extensions and generic points}

We begin with some general definitions. Let $(Y,T,\mu)$ be a measure-preserving dynamical system for a compact metric space $Y$ together with its Borel $\sigma$-algebra, $G$ a compact group endowed with the Haar measure $m_G$ and $\alpha: Y \rightarrow G$ a continuous map. Define the \textit{skew-product} map $S$ on $X= Y \times G$ by the formula
$$
S(y,g) = (T(y),\alpha(y)\cdot g).
$$ A measure $\vartheta$ on $X$ is said to \textit{project over }$\mu$ if $\pi_Y \vartheta = \mu$, where $\pi_Y$ stands for the projection on $Y$. Furthermore, $\vartheta$ is said to be \textit{uniquely ergodic over} $\mu$ if it is the unique ergodic measure which projects over $\mu$. Notice that the measure $\mu \times m_G$ clearly projects over $\mu$ and is also $S$-invariant by the Fubini-Tonelli theorem, since $\mu$ is $T$-invariant and $m_G$ is the Haar measure on $G$.

\begin{proposition}\label{equivalencia} A measure $\vartheta$ on $X$ is uniquely ergodic over an ergodic measure $\mu$ if and only if it is the unique $S$-invariant measure which projects over $\mu$.
\end{proposition}

\begin{proof} For the $\Leftarrow$ implication, we note that $\vartheta$ must be ergodic: if $\vartheta = \int \sigma\,d\rho(\sigma)$ is the ergodic decomposition of $\vartheta$ (see e.g. \cite[Theorem 6.1]{EinWar11}) then $\mu=\int \pi_Y\sigma \,d\rho(\sigma)$. Since $\mu$ is ergodic, $\pi_Y\sigma=\mu$ for $\rho$-almost all $\sigma$, whence $\sigma=\vartheta$ for $\rho$-almost all $\sigma$, showing that $\vartheta$ is ergodic as claimed. The $\Leftarrow$ implication is now clear.

Thus, suppose that $\vartheta$ is uniquely ergodic and let $\zeta$ be an $S$-invariant measure that projects over $\mu$. We must show that $\zeta=\vartheta$. Let
\[
\zeta = \int \sigma \;d\rho(\sigma)
\]
be the ergodic decomposition of $\zeta$. Hence, to see that $\zeta=\vartheta$ it suffices to show that $\rho = \delta_{\vartheta}$.
Notice that
\begin{equation}
 \mu = \pi_Y \zeta = \int \pi_Y \sigma \; d\rho(\sigma) = \int_{\mathcal{D}} \sigma' \; d \pi_Y \rho(\sigma')
\end{equation}
where $\mathcal{D}=\{ \sigma' \in \mathcal{P}(Y) : \sigma' \mbox{ is } T\mbox{-invariant} \}$.

Since $\mu$ is ergodic, it is an extreme point of the set $\mathcal{D}$ whence, by Bauer's characterization of extreme points \cite[Chapter IX, Theorem 3]{Die84}, we have that $\pi_Y \rho = \delta_{\mu}$. But then, since $\rho$ is supported on the set of ergodic measures and $\vartheta$ is uniquely ergodic, we obtain that
\begin{align*}
1 &= \pi_Y \rho ( \{ \mu \}) = \rho \big( \{ \sigma: \pi_Y \sigma = \mu \} \big) \\
&=  \rho \big( \{ \sigma   : \sigma \mbox{ is ergodic and } \pi_Y \sigma = \mu \} \big) = \rho (\{ \vartheta\} )
\end{align*}
which implies that $\rho = \delta_{\vartheta}$ and concludes the proof.
\end{proof}

Recall that given a measure-preserving system $(Z,R,\sigma)$, with $Z$ compact and $R$ continuous, we say that $z \in Z$ is \emph{generic} for $\sigma$ (or $\sigma$-generic) if for any continuous function $f: Z \rightarrow \R$ we have
\[
\lim_{n \to \infty}  \frac{1}{n} \sum_{i=0}^{n-1} f(R^i(z)) = \int_Y f \;d\sigma.
\]
It follows from the ergodic theorem that if $\sigma$ is ergodic then $\sigma$-almost every $z \in Z$ is generic for $\sigma$.

The special case of the next lemma in which  the base system $(Y,T)$ is uniquely ergodic is a classical result of Furstenberg, see e.g.  \cite[Theorem 4.21]{EinWar11}. The general case goes along the same lines and is surely known, but we give the proof for completeness.
\begin{lemma}\label{unicamenteergodica} If $\vartheta = \mu \times m_G$ is ergodic then it is uniquely ergodic over $\mu$.
\end{lemma}

\begin{proof} Clearly we have $\pi_Y \vartheta = \mu$, so that it remains to check that $\vartheta$ is the unique ergodic measure with this property.
Now, since $\vartheta$ is ergodic we have that $\vartheta$-almost every $(\tilde{\omega},\tilde{g}) \in X$ is generic for $\vartheta$. Furthermore, we have that
\begin{equation}\label{implicacion}
(\omega,g) \text{ is generic for $\vartheta$} \Longrightarrow (\omega,g') \text{ is generic for $\vartheta$ for every $g' \in G$}.
\end{equation} Indeed, observe that for any $i \in \mathbb{N}_0$ we have that
$$
S^i(\omega,g') = M_{g^{-1}\cdot g'}(S^i(\omega,g))
$$ where for any $h \in G$ the map $M_h : X \rightarrow X$ is defined by the formula
$$
M_{h}(\tilde{\omega},\tilde{g}) := (\tilde{\omega},\tilde{g}\cdot h).
$$ Then for any continuous function $f$ we have that
$$
\frac{1}{n} \sum_{i=0}^{n-1} f(S^i (\omega,g')) = \frac{1}{n}\sum_{i=0}^{n-1} f(M_{g^{-1}\cdot g'}(S^i(\omega,g))) \longrightarrow \int_X f \circ M_{g^{-1} \cdot g'} \;d(\mu \times m_G)
$$ since $(\omega,g)$ is generic for $\vartheta$ and $f \circ M_h$ is continuous for any $h \in G$. Now, since $m_G$ is invariant under multiplication, by the Fubini theorem we conclude that
$$
\frac{1}{n} \sum_{i=0}^{n-1} f(S^i (\omega,g')) \longrightarrow \int_X f \; d(\mu \times m_G)
$$ which shows that $(\omega,g')$ is generic for $\vartheta$.

Now, let $\rho$ be an ergodic measure on $X$ which projects over $\mu$. For any such $\rho$ the set
\begin{equation}\label{lambda}
\Lambda_\rho = \{ \omega \in Y : (\omega,g) \text{ is generic for $\rho$ for some $g \in G$}\}
\end{equation} has full $\mu$-measure on $Y$. In particular, the set $\Lambda_\vartheta \cap \Lambda_\rho$ is nonempty, where $\Lambda_\vartheta$ is defined by analogy with \eqref{lambda}.
Notice that by \eqref{implicacion} we have that for any $\omega \in \Lambda_\vartheta \cap \Lambda_\rho$ there exists $g \in G$ such that $(\omega,g)$ is generic for both $\vartheta$ and $\rho$. But by definition of generic point this implies that
$$
\int_X f \; d\vartheta  = \int_X f \; d\rho
$$ for any continuous function, which shows that $\vartheta=\rho$.
\end{proof}

We finish this section with the following uniform convergence result, which again is classical in the uniquely ergodic case.

\begin{lemma}  \label{lem:uniform-convergence}
 If $\mu\times m_G$ is ergodic, then for every continuous function $f:X\to\R$ and every $\mu$-generic point $\omega$, the ergodic averages
 \[
  \frac1n \sum_{i=1}^n f(S^i(\omega,g))
 \]
 converge to $\int f\,d(\mu\times m_G)$, uniformly in $g\in G$.
\end{lemma}
\begin{proof}
Suppose the statement does not hold for some $\mu$-generic $\om$  and continuous $f$. Then we can find $\eps>0$, a sequence $n_j\to\infty$ and points $g_j\in G$ such that
\begin{equation}  \label{eq:contradiction}
 \left|\frac{1}{n_j} \sum_{i=1}^{n_j} f(S^i(\omega,g_j)) -  \int f\,d(\mu\times m_G)\right| > \eps.
\end{equation}
After passing to a further subsequence, we may assume that $\nu_j:=\tfrac{1}{n_j}\sum_{i=1}^{n_j} \delta_{S^i(\omega,g_j)}$ converges  to a measure $\vartheta$. Note that, for any $h\in C(X)$,
\[
\left|\int h\, d\nu_j - \int h\, d(S\nu_j)\right| = \frac{1}{n_j}| h(\omega,g_j) - h(S^{n_j+1}(\omega,g_j))|\le \frac{2\|h\|_\infty}{n_j},
\]
which tends to $0$ as $j\to\infty$. Thus the limit $\vartheta$ is $S$-invariant. Moreover, $\pi_Y\vartheta$ is the limit of $\tfrac{1}{n_j}\sum_{i=1}^{n_j} \delta_{T^i(\omega)}$, which equals $\mu$ thanks to our assumption that $\omega$ is generic.

It now follows from Proposition \ref{equivalencia} and Lemma \ref{unicamenteergodica} that $\vartheta=\mu\times m_G$, and hence $\tfrac{1}{n_j}\sum_{i=1}^{n_j} f(S^i(\omega,g_j))$ converges to $\int f \, d\vartheta=\int f\,d(\mu\times m_G)$. This contradicts \eqref{eq:contradiction}, as desired.
\end{proof}

\subsection{Subadditive cocycles and generic points}

It is well known that if $(X,S)$ is uniquely ergodic, then ergodic averages of continuous functions converge uniformly. This fails for cocycles over uniquely ergodic systems, but one side of the inequality holds: this was observed by Furman \cite[Theorem 1]{Fur97}. An inspection of the proof of the subadditive  ergodic theorem given by Katznleson and Weiss \cite{KatWie82} yields a more general result. First, we introduce a definition.

\begin{definition} \label{def:C-approx}
A function $\phi:X\to\R$, where $(X,\mu)$ is a measured metric space, is said to be \emph{upper $C$-approximable} if, for every $\eps>0$, there exists a continuous function $\phi_\eps:X\to\R$ such that $\phi \le \phi_\eps$ pointwise, and $\int (\phi_\eps-\phi)\,d\mu<\eps$.
\end{definition}

\begin{theorem}\label{furman} Let $(X,S,\mu)$ be an ergodic measure-preserving system with $X$ compact and $S$ continuous. Let $\mathcal{F}=(\phi_n)_{n \in \mathbb{N}}$ be an upper $C$-approximable subadditive cocycle on $X$, that is,
\begin{equation}\label{cocycledefi}
\phi_{n+m}(x) \leq \phi_n(x) + \phi_m \circ S^n(x) \quad\text{for all }x\in X.
\end{equation}
Then for any $\mu$-generic $x\in X$, it holds that
\[
\limsup_{n \rightarrow +\infty} \frac{\phi_n(x)}{n} \leq \Phi(\mathcal{F}):= \inf_{n \in \mathbb{N}}\left\{ \frac{1}{n} \int_X \phi_n \;d\mu\right\}.
\]
\end{theorem}

\begin{proof}
We start by noting that upper $C$-approximable functions are bounded above and integrable by definition. Fix $N \in \mathbb{N}$, $\varepsilon >0$ and let
\[
L := \max_{1 \leq i \leq N } \sup_{x\in X} \phi_i(x) < +\infty.
\]
By assumption, there exists a continuous function $\phi_{N,\eps}: X \to \R$ such that $\phi_N \leq \phi_{N,\eps}$ and $\int_{X} (\phi_{N,\eps}-\phi_N) d\mu < \varepsilon$.

Suppose that $n = (m+1)N +1$, for some $m \in \mathbb{N}$. Then we can write $n$ as $i +m N + (N + 1 -i)$ for every $ 1 \leq i \leq N$. By subadditivity,
\begin{align}
\phi_n(x) = \phi_{(m+1)N +1}(x) & \leq \phi_i(x) + \phi_{m N + (N + 1 -i)}(S^i (x)) \nonumber \\
& \leq \phi_i(x) +  \phi_{m N }(S^i (x)) + \phi_{N + 1 -i}(S^{i+mN} (x)) \nonumber \\
& \leq 2 L + \phi_{m N }(S^i (x)) \label{descocycle}.
\end{align}
Note also that $$\phi_{m N }(S^i (x)) \leq \phi_{N}(S^i (x)) +  \phi_{N}(S^{i+N} (x)) + \dots + \phi_{N}(S^{i+(m-1)N} (x)).$$
Therefore, if we sum over all $i$ ($1 \leq i \leq N$) in Equation \eqref{descocycle} and use the last inequality, we obtain
\begin{equation} \label{desgcocycle2}
N \phi_n(x) = N \phi_{(m+1)N +1}(x) \leq 2 L N + \sum_{j=1}^{mN} \phi_N(S^j (x)).
\end{equation}
Now, if $n=(m+1)N+1+r-1$ for some $2 \leq r \leq N$ then notice that
$$
\phi_n(x) \leq \phi_{(m+1)N +1}(x) + \phi_{r-1}(S^{(m+1)N +1}(x) ) \leq \phi_{(m+1)N +1}(x) + L.
$$
By \eqref{desgcocycle2} we obtain
$$
N \phi_n(x)\leq 3 L N + \sum_{j=1}^{mN} \phi_N(S^j (x))
$$
for all sufficiently large $n \in \mathbb{N}$, where we have written $n$ as $(m+1)N+r$ for some $m \in \N$ and $1 \leq r \leq N$. Dividing the above inequality by $Nn$ yields
\begin{equation}  \label{eq:upper-bound-fN-ergodic-avg}
\frac{\phi_n(x)}{n} \leq \frac{3L}{n} + \frac{m}{n}\Big(\frac{1}{mN}\sum_{j=1}^{mN} \phi_N(S^j (x)) \Big) \leq \frac{3L}{n} + \frac{m}{n}\Big(\frac{1}{mN}\sum_{j=1}^{mN} \phi_{N,\eps}(S^j (x)) \Big).
\end{equation}
Observe that $\frac{m}{n} \rightarrow \frac1N$ as $n \to \infty$, so that by taking $\limsup_{n \to \infty}$ in both sides and using the fact that $x$ is generic, we may conclude that
\begin{equation}\label{cotacocycle}
\limsup_{n \rightarrow +\infty} \frac{\phi_n(x)}{n} \leq \frac{1}{N} \left(\int_X \phi_N \;d\mu +\eps\right).
\end{equation}
Since the bound in \eqref{cotacocycle} holds for arbitrary $N \in \N$ and $\eps>0$, we obtain the result.
\end{proof}

In the special case of a compact group skew-product $(Y\times G, S, \mu\times m_G)$, we can apply Lemma \ref{lem:uniform-convergence} to get the following improvement.
\begin{corollary} \label{cor:uniform-in-fiber}
Let $(Y,T,\mu)$ a m.p.s. with $Y$ compact and $T$ continuous, let $G$ be a compact group with Haar measure $m_G$, and let $\alpha:Y\to G$ be continuous function. Denote the associated skew-product m.p.s. by $(X,S,\mu\times m_G)$. Furthermore, let $\mathcal{F}=(\phi_n)_{n=1}^\infty$ be a an upper $C$-approximable subadditive cocycle over $X$.

If $\mu\times m_G$ is ergodic, then for every $\mu$-generic point $\omega\in X$,
\[
\limsup_{n \rightarrow +\infty} \frac{\phi_n(\omega,g)}{n} \leq \inf_{n \in \mathbb{N}}\left\{ \frac{1}{n} \int_X \phi_n \;d(\mu\times m_G)\right\}\quad\text{uniformly in }g\in G.
\]
\end{corollary}
\begin{proof}
This follows from Lemma \ref{lem:uniform-convergence} by letting $n\to\infty$ in the pointwise bound \eqref{eq:upper-bound-fN-ergodic-avg}.
\end{proof}

\section{Proof of Theorem \ref{thm:main-projections}}

\subsection{Notation and preliminaries}
Recall that $Y=\{1,\ldots,N\}^\N$. For each $n \in \N$ let us define the $n$-th rotation $\mathbf{R}^n : Y\times S^1 \rightarrow S^1$ by the formula
$$
\mathbf{R}^n (\omega,v) := \pi_{S^1} \left( S^n (\omega,v) \right) =  \alpha(\om)\cdots \alpha(\bt^{n-1}\om)v,
$$
where $\pi_{S^1}$ denotes the projection from $Y\times S^1$ onto $S^1$.

It is easy to see that for $u \in \mathbb{X}^{(\omega)}_n$ we can decompose $f^{(\omega)}_{u}$ as
\begin{equation}\label{eq:composition-f}
f^{(\omega)}_{u}(x,y) = \lambda_{\omega_1} \cdots  \lambda_{\omega_{n}} \mathbf{R}^n(\omega,(1,0))\cdot (x,y) + d^{(\omega)}_u
\end{equation}
for a certain constant $d^{(\omega)}_u \in \R^2$ (here and it what follows we identify $\R^2$ with $\mathbb{C}\supset S^1$). Moreover, if $F^{(\omega)}_{u}$ denotes the inverse of $f^{(\omega)}_{u}$, then from \eqref{eq:composition-f} we obtain
\[
F^{(\omega)}_{u}(x,y) = \frac{\overline{\mathbf{R}^n(\omega,(1,0))}}{\lambda_{\omega_1} \cdots \lambda_{\omega_{n}}} \cdot \left( (x,y) - d^{(\omega)}_u \right).
\]

For each $\omega \in Y$ denote the projected measure $\Pi_v \eta^{(\omega)}$ by $\eta^{(\omega)}_v$, i.e. $\eta^{(\omega)}_v(B) = \eta^{(\omega)}(\Pi_v^{-1}(B))$ for every Borel set $B \subseteq \R$.

Also, for each $n \in \N$ we define $L^{(\omega)}_n$ as the unique nonnegative integer such that
\begin{equation}\label{eq:defln}
2^{-L^{(\omega)}_n} \leq \lambda_{\omega_1} \cdots   \lambda_{\omega_n} < 2^{1 -L^{(\omega)}_n},
\end{equation} and consider the family of intervals $\mathfrak{D}_n^{(\omega)}$ given by
\[
\mathfrak{D}^{(\omega)}_n= \mathfrak{D}_{L^\pom_n} = \{ [2^{-L^{(\omega)}_n}j,2^{-L^{(\omega)}_n}(j+1)) : j \in \mathbb{Z} \}.
\]
Notice that for each $\omega \in Y$ the families $\mathfrak{D}_n^\pom$ are nested: for every $n \in \N$, each element of $\mathfrak{D}^{(\omega)}_{n+1}$ is a subinterval of exactly one element of $\mathfrak{D}^{(\omega)}_n$. With this, for $q > 1$ we define the functions $\tau_{q,n}: Y\times S^1 \rightarrow \R$
\begin{equation}\label{eq:taudefi}
\tau_{q,n}(\omega,v) : = \sum_{I \in \mathfrak{D}^{(\omega)}_{n}} \left(\eta^{(\omega)}_v (I)\right)^q = \mathcal{C}^q_{\eta^{(\omega)}_v} (L^{(\omega)}_n),
\end{equation}
recall  \eqref{eq:def-Cq}.

To conclude these preliminaries, we state the version of Marstrand's Projection Theorem we alluded to in the introduction, due to Hunt and Kaloshin \cite[Theorem 1.1]{HuntKaloshin97}.
\begin{theorem} \label{thm:Marstrand}
 Let $\eta$ be a Borel probability measure on $\R^2$. If $q\in (1,2]$, then
 \[
  \underline{D}_q(\Pi_v\eta) = \min(\underline{D}_q\eta,1) \quad\text{for almost all }v\in S^1.
 \]
\end{theorem}

\subsection{A submultiplicative cocycle} \label{subsec:mult-cocycle}
Our aim  is to show that given $q >1$ there exists a continuous subadditive cocycle $\F_q=(\phi_{q,n})_{n \in \N}$ such that
\begin{equation}\label{eq:cdimensionq}
\liminf_{n \rightarrow +\infty} \frac{ \phi_{q,n}(\omega,v)}{-L^{(\omega)}_n} = \underline{D}^q( \eta^{(\omega)}_v),\quad \limsup_{n \rightarrow +\infty} \frac{ \phi_{q,n}(\omega,v)}{-L^{(\omega)}_n} = \overline{D}^q( \eta^{(\omega)}_v),
\end{equation}
for every $(\omega,v) \in Y \times S^1$. We do this in two steps. We first show that the family $(\log \tau_{q,n})_{n \in \N}$ for $\tau_{q,n}$ defined in \eqref{eq:taudefi} constitutes, up to additive constants, a subadditive cocycle. Then, we prove that there exists a ``smooth'' analogue $\overline{\tau}_{q,n}$ of $\tau_{q,n}$ which is continuous. From these facts it will follow that the cocycle $\F_q=\log \overline{\tau}_{q,n}$ enjoys the desired properties.

In this section we establish the core of this program, by showing that there exists $K_1 > 1$ such that for any $n,m \in \N$ and $(\omega,v) \in Y \times S^1$ one has
\begin{equation}\label{eq:mcocycle}
\tau_{n+m}(\omega,v) \leq K_1 \,\tau_n(\omega,v) \,\tau_{m}(S^n(\omega,v)),
\end{equation}
where we have suppressed $q$ from the notation for simplicity. This implies that the family $(\log K_1 \tau_n)_{n \in \N}$ is a subadditive cocycle.

We begin by introducing a definition.
\begin{definition} Given $M \in \N$, we say that two families $\mathfrak{P}, \mathfrak{P}'$ of sets are $M$-equivalent on $W$ (or simply $M$-equivalent) if
\begin{enumerate}
\item [(i)] $W\cap \bigcup_{A \in \mathfrak{P}} A =  W\cap \bigcup_{B \in \mathfrak{P'}} B$.
\item [(ii)] Each element of $\mathfrak{P}$ intersects at most $M$ elements of $\mathfrak{P}'$ and viceversa.
\end{enumerate}
\end{definition}

The following simple consequence of H\"{o}lder's inequality is proved in \cite[Lemma 5.3]{ShmerkinSolomyak15}.
\begin{lemma}\label{lem:equivalent-families} If $\mathfrak{P}$ and $\mathfrak{P}'$ are $M$-equivalent on $W$, and $\rho$ is a probability measure with $\rho(W)=1$, then
$$
M^{1-q} \sum_{B \in \mathfrak{P}'} \rho(B)^q \leq \sum_{A \in \mathfrak{P}} \rho(A)^q \leq M^{q-1} \sum_{B \in \mathfrak{P}'} \rho(B)^q.
$$
\end{lemma}

Now, let us fix $n,m \in \N$, $(\omega,v) \in Y \times S^1$ and proceed to show \eqref{eq:mcocycle}. Recall that $\Delta_\om:\X^\pom\to \C^\pom$ is the coding map, and let $\opi_v:\X^\pom\to\R$ denote the composition $\Pi_v\circ\Delta_\om$.  Given $J \in \mathfrak{D}^{(\omega)}_n$ let us define
\begin{equation} \label{eq:def-J-family}
\mathbb{X}_{J}^{(\omega)}(v):=\{ u \in \mathbb{X}^{(\omega)}_n : [u]_\om \cap \opi_v^{-1}(J) \neq \emptyset \},
\end{equation}
and consider the interval $\widehat{J}$ which has the same center $x_J$ as $J$ but whose length is $|\widehat{J}|= 9|J|$, i.e.
\[
\widehat{J} = \left[x_J-\frac{9}{2}|J|, x_J + \frac{9}{2}|J|\right).
\]
It is not difficult to see that, by choosing $\widehat{J}$ in this way, one has $\opi_v[u]_\om\subset \widehat{J}$ for every $u \in \mathbb{X}_{J}^{(\omega)}(v)$.
Also, note that if $u\in\X_n^\pom$ and $y\in\X_\infty^{(\bt^n \om)}$, then $\Delta_\om(uy)=f_u^\pom(\Delta_{\bt^n \om}(y))$, and therefore
\begin{equation} \label{eq:symbolic-semi-conjugacy}
 \Delta_\om^{-1}(f_u^\pom A) \cap [u]_\om = \left\{ uy: y\in \Delta_{\bt^n \om}^{-1}(A)\right\}.
\end{equation}

If $I \in \mathfrak{D}_{n+m}^{(\omega)}$ is such that $I \subset J$, then
\begin{align}
\eta^{(\omega)}_{v}(I) & = \sum_{u \in \mathbb{X}_n^{(\omega)}} \oeta^{(\omega)} \big( [u]_\om \cap \opi_v^{-1}(I) \big)   \nonumber \\
&= \sum_{u \in \mathbb{X}_{J}^{(\omega)}(v)} \oeta^{(\omega)}|_{[u]_\om}\left(\opi_v^{-1}(I) \right) & \text{ by \eqref{eq:def-J-family}} \nonumber\\
&= \sum_{u \in \mathbb{X}_{J}^{(\omega)}(v)} \oeta^{(\omega)}|_{[u]_\om} \left( \Delta_\om^{-1} f^\pom_u F^\pom_u\Pi_v^{-1}(I) \right)\nonumber\\
&=\sum_{u \in \mathbb{X}_{J}^{(\omega)}(v)} \oeta^{(\omega)} \bigl( \{ uy: y\in\Delta_{\bt^n\om}^{-1}F^\pom_u\Pi_v^{-1}(I) \}\bigr) &\text{ by \eqref{eq:symbolic-semi-conjugacy}} \nonumber\\
& \leq K  \sum_{u \in \mathbb{X}_{J}^{(\omega)}(v)} \oeta^{(\omega)}([u]_\om) \; \bigl(\oeta^{(\mathbf{T}^{n}\om)}  \Delta_{\bt^n\om}^{-1}(F^\pom_u\Pi_v^{-1}(I))\bigr) &\text{ by \eqref{eq:measure-submultiplicative}.}\label{eq1}
\end{align}
Now, observe that, by definition of $\mathbb{X}_{J}^{(\omega)}(v)$,
\begin{equation}\label{eq2}
\sum_{u \in \mathbb{X}_{J}^{(\omega)}(v)} \oeta^{(\omega)}([u]_\om) \leq \eta^{(\omega)}_v (\widehat{J}).
\end{equation}
Therefore, using \eqref{eq1}, \eqref{eq2} and H\"{o}lder we obtain that
\begin{align*}
\left(\eta^{(\omega)}_v(I)\right)^q & \leq  \left( K \sum_{u \in \mathbb{X}_{J}^{(\omega)}(v)} \oeta^{(\omega)}([u]_\om) \; \eta^{(\mathbf{T}^{n}(\omega))} \left( F^{(\omega)}_{u} \Pi_v^{-1} (I) \right)  \right)^q  & \\
& \leq K^q \left( \sum_{u \in \mathbb{X}_{J}^{(\omega)}(v)} \oeta^{(\omega)}([u]_\om) \right)^{\frac{q}{q'}} \times\\
& \qquad \times \sum_{u \in \mathbb{X}_{J}^{(\omega)}(v)}
\oeta^{(\omega)}([u]_\om) \left(\eta^{(\mathbf{T}^{n}(\omega))}\left( F^{(\omega)}_{u}\Pi_v^{-1} (I)\right)\right)^q \\
& \leq K^q \big(\oeta^{(\omega)}_v(\widehat{J}) \big)^{q-1} \sum_{u \in \mathbb{X}_{J}^{(\omega)}(v)}
\oeta^{(\omega)}([u]_\om) \left(\eta^{(\mathbf{T}^{n}(\omega))} \left( F^{(\omega)}_{u} \Pi_v^{-1} (I)\right) \right)^q.
\end{align*}
Summing over all $I \in \mathfrak{D}^{(\omega)}_{n+m}$ such that $I \subset J$, we get
\begin{equation} \label{eq:submult-interm-step}
\sum_{I \in \mathfrak{D}^{(\omega)}_{n+m} \atop I \subset J}  \left(\eta^{(\omega)}_v(I)\right)^q \leq K^q \left(\eta^{(\omega)}_v(\widehat{J})\right)^{q-1} \Lambda_v(J),
\end{equation}
where
\begin{align}
\Lambda_v(J)&=  \sum_{I \in \mathfrak{D}^{(\omega)}_{n+m} \atop I \subset J} \sum_{u \in \mathbb{X}^{(\omega)}_{J}(v)}\oeta^{(\omega)}([u]_\om) \left(\eta^{(\mathbf{T}^{n}(\omega))} \big( F^{(\omega)}_{u}\Pi_v^{-1} (I) \big)\right)^q\nonumber\\
&= \sum_{u \in \mathbb{X}^{(\omega)}_{J}(v)}\oeta^{(\omega)}([u]_\om)\sum_{I \in \mathfrak{D}^{(\omega)}_{n+m} \atop I \subset J} \left(\eta^{(\mathbf{T}^{n}(\omega))} \big( F^{(\omega)}_{u}\Pi_v^{-1} (I) \big)\right)^q. \label{eq:exchanged-sum}
\end{align}

Now, using \eqref{eq:composition-f} it is not hard to see that for any such interval $I$ and $u \in \mathbb{X}^{(\omega)}_J(v)$ one has
\begin{equation}\label{eq:expression-inverse}
F^{(\omega)}_{u}\Pi_v^{-1} (I) = \left(\Pi_v \circ f^{(\omega)}_{u}\right)^{-1}(I) = \Pi_{v'}^{-1} \left(\frac{1}{\lambda_{\omega_1} \cdots \lambda_{\omega_{n}}} \cdot (I - \Pi_v(d^{(\omega)}_u))\right)
\end{equation}
where $v':= \mathbf{R}^n(\omega,v)$.  Write $\ell=\lambda_{\omega_{n+1}}\cdots \lambda_{\omega_{n+m}}$, and note that the family
\[
\left\{ \frac{1}{\lambda_{\omega_1} \cdots \lambda_{\omega_{n}}} (I - \opi_v( d^{(\omega)}_u)) : I \in \mathfrak{D}^{(\omega)}_{n+m}\right\}
\]
is composed of consecutive intervals of equal length between $\tfrac12\ell$ and $\ell$. Since the same is true for the family $\mathfrak{D}^{(\bt^n\omega)}_{m}$, these families are $6$-equivalent. It follows from Lemma \ref{lem:equivalent-families} and \eqref{eq:expression-inverse} that
\begin{equation}\label{sumareversa}
\sum_{I \in \mathfrak{D}^{(\omega)}_{n+m} \atop I \subset J} \left(\eta^{(\mathbf{T}^{n}(\omega))} \big( F^{(\omega)}_{u}\Pi_v^{-1} (I) \big)\right)^q \leq 6^{q-1} \tau_m(\mathbf{T}^{n}(\omega),v') = 6^{q-1} \tau_m( S^n(\omega,v)).
\end{equation}
Combining  \eqref{eq2}, \eqref{eq:submult-interm-step},  \eqref{eq:exchanged-sum} and \eqref{sumareversa} yields
\begin{equation}\label{sumaprefinal}
\sum_{I \in \mathfrak{D}^{(\omega)}_{n+m} \atop I \subset J}  \left(\eta^{(\omega)}_v(I)\right)^q \leq (6K)^q \left(\eta^{(\omega)}_v(\widehat{J})\right)^{q}\tau_m( S^n(\omega,v)).
\end{equation}
Finally, by summing \eqref{sumaprefinal} over all $J \in \mathfrak{D}^{(\omega)}_n$, we conclude
\begin{align*}
\tau_{n+m}(\omega,v) &= \sum_{J \in \mathfrak{D}^{(\omega)}_n} \sum_{I \in \mathfrak{D}^{(\omega)}_{n+m} \atop I \subset J} \left(\eta^{(\omega)}_v(I)\right)^q \\
  & \leq (6K)^q \tau_m( S^n(\omega,v)) \sum_{J \in \mathfrak{D}^{(\omega)}_n} \left(\eta^{(\omega)}_v(\widehat{J})\right)^q\\
  & \leq (54K)^q \tau_n(\omega,v) \tau_m( S^n(\omega,v))
\end{align*}
where to obtain the last inequality we used Lemma \ref{lem:equivalent-families} applied to the families $\mathfrak{D}^{(\omega)}_n$ and $\{ \widehat{J} : J \in \mathfrak{D}^{(\omega)}_n\}$, which are $9$-equivalent. This gives \eqref{eq:mcocycle} for $K_1:= (54K)^q$.

\subsection{A continuous analog of $\tau_{n}$} \label{subsec:continuous-analog} We now construct for each $n \in \N$ a continuous function $\overline{\tau}_n$ which is comparable up to multiplicative constants to $\tau_n$. To this end, we consider $\psi \in C^\infty_0(\R)$ supported on $[-2,2)$ such that $0 \leq \psi \leq 1$ and $\psi|_{[-1,1)} \equiv 1$. For each $n \in \N$ and $\omega \in Y$ define $\psi_n^{(\omega)} : \R^2 \rightarrow \R$ by the formula
$$
\psi_n^{(\omega)}(x,y) = \psi ( 2^{L_n^{(\omega)}}(x - y) ).
$$
For any fixed $y \in \R$ the function $\psi_{n,y}^{(\omega)}(x):= \psi_n^{(\omega)}(x,y)$ is supported on the interval $[y - 2^{1-L_n^{(\omega)}}, y + 2^{1-L_n^{(\omega)}})$ and is equal to 1 on the interval $[y-2^{-L_n^{(\omega)}}, y + 2^{-L_n^{(\omega)}})$. Define $\overline{\tau}_n:Y\times S^1 \rightarrow \R$ by the formula
\[
\overline{\tau}_n(\omega,v):= \int_\R \left( \int_\R \psi_n^{(\omega)}(x,y) d\eta_v^{(\omega)}(x) \right)^{q-1} d\eta_v^{(\omega)}(y).
\]
Notice that $\overline{\tau}_n$ can be rewritten as
\[
\overline{\tau}_n(\omega,v)= \int_{\mathbb{X}_\infty} \left(\Psi_n^{(\omega)}(u,v)\right)^{q-1}d \oeta^{(\omega)}(u),
\]
where
\[
\Psi_n^{(\omega)}(u,v):=\int \psi_n^{(\omega)}(\overline{\Pi}_v(u) , \overline{\Pi}_v(u')) \,d\oeta^{(\omega)}(u').
\]
We claim that $\overline{\tau}_n$ is continuous. Indeed, this is a consequence of the following fact.
\begin{lemma} \label{lem:continuity}
 Let $X$ be a compact metric space and let $\{ \rho^\pom: \om\in X\}$ be a family of Borel probability measures on some other compact metric space $Z$, such that $\om\mapsto\rho^\pom$ is continuous. Then, for any continuous function $h:X\times Z\to\R$, the function $\om\mapsto \int h(\om,u)\,d\rho^\pom(u)$ is continuous.
\end{lemma}
\begin{proof}
By uniform continuity, given $\eps>0$ we have $|h(\om,u)-h(\om',u)|<\eps$ provided $d(\om,\om')$ is small enough. It follows that
\begin{align*}
 \limsup_{\om'\to\om} & \left|\int h(\om,u)\,d\rho^\pom(u)- \int h(\om',u)\,d\rho^{(\om')}(u)\right| \le\\
    & \eps +\lim_{\om'\to\om}  \left|\int h(\om,u)\,d\rho^\pom(u)- \int h(\om,u)\,d\rho^{(\om')}(u)\right| = \eps.
\end{align*}
\end{proof}

Note that $L_n(\om)$ is continuous, since it depends only on the first $n$ coordinates of $\om$, and hence $\psi_n^\pom(x)$ is jointly continuous. We can then apply Lemma \ref{lem:continuity} with $X=Y\times \X_\infty\times S^1$, $Z=\X_\infty$, $\rho^{(\om,u,v)}=\overline{\eta}^\pom$, and $h((\om,u,v),u')= \psi_n^{(\omega)}(\overline{\Pi}_v(u) , \overline{\Pi}_v(u'))$ to obtain that $\Psi_n^{(\omega)}(u,v)$ is continuous in $(\omega,u,v)$. A second application of Lemma \ref{lem:continuity} yields the continuity of $\overline{\tau}_n$.

It remains to see that $\overline{\tau}_n$ is equivalent to $\tau_n$, i.e. there exists $M > 1$ such that
\begin{equation} \label{eq:equivalence}
M^{-1}\tau_n \leq \overline{\tau}_n \leq M \tau_n.
\end{equation}
To show the rightmost inequality we notice that for any $(\omega,v) \in Y\times S^1$ we have that
\begin{align*}
\overline{\tau}_n(\omega,v) &= \sum_{j \in \Z} \int_{j2^{-L_n^{(\omega)}}}^{(j+1)2^{-L_n^{(\omega)}}} \left(\int_{\R} \psi_n^{(\omega)}(x ,y) d\eta^{(\omega)}_v(x)\right)^{q-1} d\eta^{(\omega)}_v(y) \\
& \leq \sum_{j \in \Z} \int_{j2^{-L_n^{(\omega)}}}^{(j+1)2^{-L_n^{(\omega)}}} \left(\eta^{(\omega)}_v([y - 2^{1-L_n^{(\omega)}}, y + 2^{1-L_n^{(\omega)}}))\right)^{q-1} d\eta^{(\omega)}_v(y)\\
& \leq \sum_{j \in \Z} \int_{j2^{-L_n^{(\omega)}}}^{(j+1)2^{-L_n^{(\omega)}}} \left(\eta^{(\omega)}_v([(j-2)2^{-L_n^{(\omega)}}, (j+3)2^{-L_n^{(\omega)}}))\right)^{q-1} d\eta^{(\omega)}_v(y)\\
& \leq \sum_{j \in \Z} \left[\eta^{(\omega)}_v([(j-2)2^{-L_n^{(\omega)}}, (j+3)2^{-L_n^{(\omega)}}))\right]^q
\end{align*}
which, upon noticing that the families $\mathfrak{D}^{(\omega)}_n$ and $\{ [(j-2)2^{-L_n^{(\omega)}}, (j+3)2^{-L_n^{(\omega)}}) : j \in \Z \}$ are $5$-equivalent, implies that $\overline{\tau}_n \leq 5^{q-1} \tau_n$.
On the other hand, to establish the leftmost inequality we observe that for any $\omega \in Y$ and $j \in \Z$ we have the inclusion
$$
[j2^{-L_n^{(\omega)}},(j+1)2^{-L_n^{(\omega)}}) \subseteq [y-2^{-L_n^{(\omega)}},y+2^{-L_n^{(\omega)}})
$$
whenever $y \in [j2^{-L^{(\omega)}_n},(j+1)2^{-L^{(\omega)}_n})$. Thus, for $(\omega,v) \in Y \times S^1$ this yields
\begin{align*}
\tau_{n}(\omega,v) & = \sum_{j \in \Z} \left(\eta^{(\omega)}_v([j2^{-L^{(\omega)}_n},(j+1)2^{-L^{(\omega)}_n}))\right)^q\\
& \leq \sum_{j \in \Z} \int_{j2^{-L^{(\omega)}_n}}^{(j+1)2^{-L^{(\omega)}_n}} \left(\eta^{(\omega)}_v([y-2^{-L_n^{(\omega)}},y+2^{-L_n^{(\omega)}}))\right)^{q-1} d\eta^{(\omega)}_v (y)\\
& \leq \sum_{j \in \Z} \int_{j2^{-L^{(\omega)}_n}}^{(j+1)2^{-L^{(\omega)}_n}} \left(\int_\R \psi_n^{(\omega)}(x ,y) d\eta^{(\omega)}_v(x)\right)^{q-1} d\eta^{(\omega)}_v (y) = \overline{\tau}_n(\omega,v)
\end{align*} which shows that
\[
\tau_n \leq \overline{\tau}_n \leq 5^{q-1} \tau_n
\]
and so both quantities are indeed equivalent. Furthermore, if we replace $\overline{\tau}_n$ by $5^{q-1}K_1\overline{\tau}_n$, where the constant $K_1$ is as in \eqref{eq:mcocycle}, then \eqref{eq:equivalence} still holds (for a different constant $M$), and
\[
\log\overline{\tau}_{n+m}(\om,v) \le \log\overline{\tau}_n(\om,v) + \log\overline{\tau}_m(S^n(\om,v)).
\]
Notice that $\log\overline{\tau}_n$ is well defined by \eqref{eq:equivalence}, since $\tau_n$ is strictly positive by its mere definition. Furthermore, each $\log\tau_n$ is continuous, since $\tau_n$ is. Thus, we conclude that the sequence $(\log\overline{\tau}_n)_{n \in \N}$ is a continuous subadditive cocycle on $Y \times S^1$.

\subsection{The proof of Equation \eqref{eq:cdimensionq}}

Write $\phi_n=\log\overline{\tau}_n$ for simplicity. We can now show that
\begin{equation}\label{correlacion}
\underline{D}^q(\eta^{(\omega)}_v) = \liminf_{n \rightarrow +\infty} \frac{\phi_n(\omega,v)}{-(q-1) L_n^{(\omega)}}
\end{equation}
for all $(\omega,v) \in Y \times S^1$, and likewise for $\overline{D}^q(\eta^{(\omega)}_v)$. It follows from the definition of $\tau_n$ and \eqref{eq:equivalence} that
\[
\left|\phi_n(\om,v)-\log\mathcal{C}_{\eta^\pom_v}^q(L_n^\pom)\right|
\]
is uniformly bounded (independent of $n$). Since $L_n^\pom\to\infty$ as $n\to\infty$, it is enough to check that
\[
 \liminf_{n\to\infty} \frac{\log\mathcal{C}_{\eta^\pom_v}^q(L_n^\pom)}{-L_n^\pom} = \liminf_{k\to\infty} \frac{\log\mathcal{C}_{\eta^\pom_v}^q(k)}{-k}.
\]
The ``$\ge$'' inequality is clear, since the limit in the left is taken along a subsequence. To see the other inequality, fix $k$ and choose $n$ such that $L_n^\pom \le k < L_{n+1}^\pom$. Note that $L_{n+1}^\pom \le L_n^\pom + \ell^*$, where $\ell^*=1+\max_{i=1}^N |\log\lambda_i|$. On the other hand, the sequence $k\mapsto \mathcal{C}_{\nu}^q(k)$ is always decreasing for $q>1$, since for a dyadic interval $J=I_1\cup I_2$ one has $\nu(J)^q\ge \nu(I_1)^q+\nu(I_2)^q$. Hence
\[
 \frac{\log\mathcal{C}_{\eta^\pom_v}^q(k)}{-k} \geq \frac{\log\mathcal{C}_{\eta^\pom_v}^q(L_n^\pom)}{-L_{n+1}^\pom} \geq  \frac{\log\mathcal{C}_{\eta^\pom_v}^q(L_n^\pom)}{\ell^*-L_{n}^\pom},
\]
provided $n$ is large enough.

The claim for $\underline{D}_q$ follows by taking a limit along an appropriate subsequence of $k$, and the case of $\overline{D}_q$ is analogous.

\subsection{A subadditive cocycle for the measures $\eta^\pom$}

The foregoing analysis of the measures $\eta^\pom_v$ has a corresponding, but simpler, correlate for the measures $\eta^\pom$. Since the proofs are very similar, we only state the results, leaving the details to the interested reader. Let $\mathfrak{Q}_n$ denote the family of dyadic squares
\[
\{  [j_1 2^{-n}, (j_1+1) 2^{-n})\times [j_2 2^{-n}, (j_2+1) 2^{-n}) : j_1, j_2\in\Z \}.
\]
Write $\mathfrak{Q}_n^\pom = \mathfrak{Q}_{L_n^\pom}$, and define
\[
\xi_n(\om) = \sum_{Q\in\mathfrak{Q}_n^\pom} \left(\eta^\pom(Q)\right)^q.
\]
Then one can check, as before, that there exists a sequence of \emph{continuous} functions $\overline{\xi}_n$, such that
\[
M^{-1}\xi_n(\om) \le \overline{\xi}_n(\om) \le M \xi_n(\om)
\]
for some constant $M>0$ (depending on $q$) and all $\om\in Y$, and furthermore
\[
\overline{\xi}_{n+m}(\om) \le  \overline{\xi}_n(\om)\, \overline{\xi}_m(\bt^n(\om)).
\]
From here one can deduce, as we have done for the projections $\eta^\pom_v$, that
\begin{equation} \label{eq:cocycle-dim-upstairs}
\underline{D}_q(\eta^{(\om)}) = \liminf_{n\to\infty} \frac{\log \overline{\xi}_n(\om)}{-L_n^\pom},\quad \overline{D}_q(\eta^{(\om)}) = \limsup_{n\to\infty} \frac{\log \overline{\xi}_n(\om)}{-L_n^\pom}.
\end{equation}

\subsection{Conclusion of the proof}

We start by applying \eqref{eq:cocycle-dim-upstairs} to show that $D_q(\eta^\pom)$ exists and is constant $\mu$-almost everywhere.  However, before we can do so it is clear that we must first understand the behavior of the quotient $\frac{L_n^{(\omega)}}{n}$ as $n$ tends to infinity. This is the purpose of the following lemma.

\begin{lemma}\label{lem:limite} If $\omega \in Y$ is $\mu$-generic, then
$$
\lim_{n \rightarrow +\infty} \frac{ -L_n^{(\omega)}}{n} = \int_{Y} \log \lambda_{\om_1} d\mu(\omega).
$$
\end{lemma}

\begin{proof} Since by definition of $L_n^{(\omega)}$ we have $2^{-L_n^{(\omega)}} < \prod_{i=1}^{n} \lambda_{\omega_i} \leq 2^{1-L_n^{(\omega)}}$ for every $n \in \N$, it suffices to show that
$$
\lim_{n \rightarrow +\infty} \frac{ \sum_{i=0}^{n-1} \log\left(\lambda_{(\mathbf{T}^i(\omega))_1} \right)}{n} = \int_{Y} \log \lambda_{\om_1} d\mu(\omega).
$$
But this follows at once from the fact that $\omega$ is $\mu$-generic, since the application $\omega \mapsto \log \lambda_{\om_1}$ is continuous on $Y$.
\end{proof}

Now it follows from the subadditive ergodic theorem that for $\mu$-almost all $\om$ it holds that
\[
 \lim_{n\to\infty} \frac{\log\overline{\xi}_n(\om)}{-L_n^\pom}  = \frac{\inf_{n\in\N} \left[ \tfrac1n \int_Y \log\overline{\xi}_n(\om)\,d\mu(\om)\right] }{\int_{Y} \log (\lambda_{\om_1}) d\mu(\omega)}=: D(q).
\]
We deduce from \eqref{eq:cocycle-dim-upstairs} that $D_q(\eta^\pom)$ exists and equals $D(q)$ for $\mu$-almost all $\om$. Furthermore, if $\om$ is $\mu$-generic, then it follows from Theorem \ref{furman} that
\[
  \liminf_{n\to\infty} \frac{\log\overline{\xi}_n(\om)}{-L_n^\pom} \ge D(q),
\]
whence $\underline{D}_q(\eta^\pom) \ge D(q)$.

Now we move onto the projections $\eta_v^\pom$. Let us begin by observing that for all $(\omega,v) \in Y\times S^1$ we have
$$
\underline{D}^q(\eta^\pom_v) \le \overline{D}^q(\eta^{(\omega)}_v) \leq \min\{ \overline{D}^q(\eta^{(\omega)}),1\}.
$$
Indeed, this follows from the well-known facts that $\overline{D}_q$ does not increase under Lipschitz maps, and can never exceed the dimension of the ambient space.

Now, since $\mu \times \mathcal{L}_\beta$ is ergodic by assumption, Corollary \ref{cor:uniform-in-fiber} combined with Lemma \ref{lem:limite} imply that any $\mu$-generic $\om$ satisfies, for each $v \in S^1$,
\begin{equation}\label{ineqcorrelacion}
\min\{ \overline{D}^q(\eta^{(\omega)}),1\} \geq \underline{D}^q(\eta^{(\omega)}_v) = \liminf_{n \rightarrow +\infty} \frac{\phi_n(\omega,v)}{-L_n^{(\omega)}} = \frac{1}{\mu^*} \liminf_{n \rightarrow +\infty} \frac{\phi_n(\omega,v)}{n} \geq \frac{\Phi}{\mu^*}
\end{equation} where
$$
\Phi := \inf_{n \in \N} \left[ \frac{1}{n} \int_{Y} \phi_n d(\mu \times \mathcal{L}_\beta)\right],\quad \mu^*:= \int_{Y} \log \lambda_{\om_1} d\mu(\omega).
$$
Hence, given any $\om\in Y$, if we wish to prove \eqref{eq:main-result}, then it suffices to show that
\begin{equation}\label{eqcorrelacion}
\min\{ \overline{D}^q(\eta^{(\omega)}),1\} = \frac{\Phi}{\mu^*}
\end{equation}

Since $(\phi_n)_{n \in \N}$ is a bounded subadditive cocycle, the subadditive ergodic theorem yields upon an application of the Fubini theorem that $\mu$-almost every $\omega \in Y$ satisfies
$$
\lim_{n \rightarrow +\infty} \frac{\phi_n (\omega,v)}{-L_n^\pom} = \frac{\Phi}{\mu^*}
$$
for $\mathcal{L}$-almost every $v \in S^1$. In light of \eqref{eq:cdimensionq}, $\frac{\Phi}{\mu^*}$ equals the $(\mu\times\mathcal{L})$-almost sure value of $D_q(\eta_v^\pom)$, and from Theorem \ref{thm:Marstrand} and Fubini, we deduce that this equals $\min(D(q),1)$ (this is the point of the proof where we use that $q\le 2$). Hence, if we let
\[
 \mathcal{E} = \{ \om\in Y: D_q(\eta^\pom) = D(q)\},
\]
then \eqref{eqcorrelacion} holds for all $\om\in\mathcal{E}$, and if $D(q)\ge 1$, also for $\om$ in the set $\mathcal{G}$ of $\mu$-generic points.

We conclude that for every $\omega$ in the full $\mu$-measure set $\mathcal{G} \cap \mathcal{E}$,  all the inequalities in \eqref{ineqcorrelacion} are equalities, and hence \eqref{eq:main-result} is satisfied. If $\mathcal{G}\subset\mathcal{E}$, or if $D(q)\ge 1$, then \eqref{eq:main-result} holds for any $\mu$-generic $\omega$.

The claim concerning uniform convergence over $v\in S^1$ follows from the above analysis, and the uniformity in Corollary \ref{cor:uniform-in-fiber} (which implies that the rightmost inequality in \eqref{ineqcorrelacion} holds uniformly in $v$, for any fixed $\om\in \mathcal{G}$). This finishes the proof of Theorem \ref{thm:main-projections}.

\section{Proof of Theorem \ref{thm:mainconvolutions}}

\subsection{Preliminaries}

The proof of Theorem \ref{thm:mainconvolutions} follows the same general outline as the proof of Theorem \ref{thm:main-projections}. We will therefore indicate where the main differences lie, and sketch or omit the parts of the proof that closely follow the arguments from Theorem \ref{thm:main-projections}.

Given $\om\in Y$, write $\oeta^\pom=\overline{\nu}^\pom\times\otheta$. For $s \in [-\beta,\beta)$ we consider the orthogonal projection $\Pi_s$ onto the linear space generated by the vector $(1,e^s)$, i.e. $\Pi_s(x,y)= x + e^sy$, and write $\opi_s=\Pi_s\circ \Delta_\om$ where, abusing notation slightly, we denote also by $\Delta_\om:\X_\infty^\pom\times \X'\to [0,1]^2$ the product coding map
\[
\Delta_\om(u,u')=(\Delta_\om(u),\Delta'(u')).
\]
For each $\omega \in Y$, denote the projected measure $\opi_s \oeta^{(\omega)}$ by $\eta^{(\omega)}_s$. Then $\eta^\pom_s$ is nothing else than the convolution $\nu^\pom* A_{e^s}\vartheta$ we are interested in.

For $(u,v) \in \X^{(\omega)}_n \times \X'_{n'}$ we define the product function $h^{(\omega)}_{u,v} : [0,1]^2 \rightarrow [0,1]^2$ by the formula
\[
h^{(\omega)}_{u,v}(x,y) := \left(f_u^{(\omega)}(x),g_v(y)\right).
\]
In the course of the proof it will be important to work with families of pairs $(u,v)$ such that the eccentricity of the rectangle $h^\om_{u,v}([0,1])^2$ is bounded, and behaves like a rotation under the action of the skew-product $\mathbf{S}$. In order to do this, we need to introduce a number of families of pairs of words. Similar families appear in \cite{NPS12}, although here we will need an extra family due to the somewhat more complicated setting.

Hence, let us consider the families $\W^{(\omega)}=(\W^{(\omega)}_n)_{n \in \N}$, $\mathbb{Y}^{(\omega)}=(\mathbb{Y}^{(\omega)}_n)_{n \in \N}$ and $\mathbb{Z}^{(\omega)}=(\mathbb{Z}^{(\omega)}_n)_{n \geq 3l}$ of word pairs defined as
\begin{align*}
\W^{(\omega)}_n &=  \X^{(\omega)}_{rn} \times \X'_{n +2l \xi^{(\omega)}_n},\\
\mathbb{Y}^{(\omega)}_n &= \X^{(\omega)}_{rn} \times \X'_{n +2l (\xi^{(\omega)}_n+1)},\\
\mathbb{Z}^{(\omega)}_n &= \X^{(\omega)}_{rn} \times \X'_{n +2l (\xi^{(\omega)}_n-1)},
\end{align*}
where $\xi^{(\omega)}_n := \#\left\{k \in \{1, \dots, n \} : \mathbf{R}^{k-1}(\omega,0) + \alpha(\mathbf{T}^{rk}(\omega)) \geq \beta\right\}$ counts the number of times $k \leq n$ for which the $k$-th rotation of the origin $0 \in S^1_\beta$ given by the term $\mathbf{R}^k(\omega,0)$ crosses the endpoint $\beta$. If we identify each word $(u,v) \in \X^{(\omega)}_n \times \X'_{n'}$ with the rectangle $Q^{(\omega)}_{u,v}:=I^{(\omega)}_u \times I_v$, obtained as the image of the function $h^{(\omega)}_{u,v}$, then we have the following properties of $\W^{(\omega)}$, $\mathbb{Y}^{(\omega)}$ and $\mathbb{Z}^{(\omega)}$:
\begin{enumerate}
\item [i.] Each rectangle of $\W^{(\omega)}$, $\mathbb{Y}^{(\omega)}$ and $\mathbb{Z}^{(\omega)}$ is the product of basic intervals of $\mathcal{C}^{(\omega)}$ and $\mathcal{C}'$ (where $\mathcal{C}'$ is the image of $\X'$ under the coding map), each of these possibly belonging to different steps in the construction of $\mathcal{C}^{(\omega)}$ and $\mathcal{C}'$, respectively.
\item [ii.] An easy calculation using \eqref{control1} and \eqref{control} shows that the size of all rectangles in $\W^{(\omega)}_n$, $\Y^{(\omega)}_n$ and $\Z^{(\omega)}_n$ is, respectively,
\begin{align*}
a_{\omega_1}\dots a_{\omega_{rn}} & \times a_{\omega_1}\dots a_{\omega_{rn}} e^{\mathbf{R}^n(\omega,0)},\\
a_{\omega_1}\dots a_{\omega_{rn}} &\times a_{\omega_1}\dots a_{\omega_{rn}} e^{\mathbf{R}^n(\omega,0)-2\beta},\\
a_{\omega_1}\dots a_{\omega_{rn}} &\times a_{\omega_1}\dots a_{\omega_{rn}} e^{\mathbf{R}^n(\omega,0)+2\beta}.
\end{align*}
In particular, the eccentricity of rectangles in each family, i.e. their height-width ratio, always stays bounded in between $e^{-3\beta}$ and $e^{3\beta}$.

\item [iii.] Under the convention $\W^{(\omega)}_{0} := \{[0,1]^2\}$, for $n \in \N_0$ the rectangles in $\W^{(\omega)}_{n+1}$ are obtained from those in $\W^{(\omega)}_{n}$ by advancing $r$ steps further in the construction of $\mathcal{C}^{(\omega)}$, and advancing either one step further in the construction of $\mathcal{C}'$ if the resulting rectangle has eccentricity in between $e^{-\beta}$ and $e^{\beta}$, or $2l+1$ steps further in the construction otherwise. By \eqref{control1}, the first option increases  the eccentricity of the resulting rectangle by a factor of $e^{\mathbf{R}^{n+1}(\omega,0)-\mathbf{R}^{n}(\omega,0)}$ with respect to the one of its predecessor in $\W^{(\omega)}_n$, whereas the second option has the effect of bringing the eccentricity of the resulting rectangle back to a value between $e^{-\beta}$ and $e^{\beta}$.

\item [iv.] The rectangles in $\mathbb{Y}^{(\omega)}_{n}$ are obtained from those in $\W^{(\omega)}_{n}$ by advancing $2l$ steps further in the construction of $\mathcal{C}'$ while keeping the same basic intervals
in the construction of $\mathcal{C}^{(\omega)}$. This yields rectangles with greater width than height.
\item [v.] The rectangles in $\mathbb{Z}^{(\omega)}_{n}$ are obtained from those in $\W^{(\omega)}_{n}$ by going $2l$ steps backwards in the construction of $\mathcal{C}'$ (notice that this is possible since $n \geq 3l$) while keeping the same basic intervals in the construction of $\mathcal{C}^{(\omega)}$. This yields rectangles with greater height than width.
\item [vi.] For each $n \in \N$, the rectangles in $\W^{(\omega)}_n$ cover the product set $\mathcal{C}^{(\omega)} \times \mathcal{C}'$ (and the symbolic rectangles are disjoint, although their geometric projections may overlap). The same statement holds for $\mathbb{Y}^{(\omega)}_n$ and $\mathbb{Z}^{(\omega)}_n$.
\end{enumerate}

From the above considerations it is easy to see that for $(u,v) \in \W^{(\omega)}_n$ we can decompose $h^{(\omega)}_{u,v}$ as
\begin{equation}\label{eq:h}
h^{(\omega)}_{u,v}(x,y) = a_{\omega_1} \dots  a_{\omega_{rn}} \big(x,e^{\mathbf{R}^n(\omega,0)}y\big) + (t^{(\omega)}_u,t_v)
\end{equation} for certain translations $t^{(\omega)}_u, t_v \in [0,1]$. Moreover, if $H^{(\omega)}_{u,v}$ denotes the inverse of $h^{(\omega)}_{u,v}$, then from \eqref{eq:h} we obtain
\[
H^{(\omega)}_{u,v}(x,y) = \frac{1}{a_{\omega_1} \cdots  a_{\omega_{rn}}} \big(x,e^{-\mathbf{R}^n(\omega,0)}y\big) - \frac{1}{a_{\omega_1} \cdots  a_{\omega_{rn}}} \big(t^{(\omega)}_u, e^{-\mathbf{R}^n(\omega,0)}t_v\big).
\]
Obviously, similar decompositions hold for $h^{(\omega)}_{u,v}$ and $H^{(\omega)}_{u,v}$ whenever $(u,v) \in \mathbb{Y}^{(\omega)}_n$ or $(u,v) \in \mathbb{Z}^{(\omega)}_n$.

We note that the family $(\oeta^{(\omega)})_{\omega \in Y}$ satisfies the following conditions, closely related to (a)-(c) above.
\begin{enumerate}
\item [(a')] Each $\oeta^{(\omega)}$ is supported on $\X_\infty^{(\om)}\times \X'_\infty$.
\item [(b')] The mapping $\omega \mapsto \overline{\eta}^{(\omega)}$ is continuous.
\item [(c')] There exists $K'' > 0$ such that
\begin{equation}\label{eq:measure-sub-product-conv}
\oeta^{(\omega)}\left([u v]_\om\times [u' v']\right)  \leq K'' \,\oeta^{(\omega)}([u]_\om\times [u']) \, \oeta^{\left(\mathbf{T}^n\omega\right)}([v]_{\mathbf{T}^n\om}\times [v']),
\end{equation}
for all $\om\in Y$, all $(u,u')\in \W_n^\pom\cup \Y_n^\pom\cup \Z_n^\pom$, and all $(v,v')\in \X_{m}^{(\bt^n\om)}\times \X'_{m'}$.

\end{enumerate}

As a matter of fact, we will prove the result for projections of families of measures $\oeta^{\pom}$ satisfying these conditions (i.e. it will not matter that $\oeta^\pom$ is a product measure for each $\om$).

\subsection{A submultiplicative  cocycle}\label{cocycle}

For each $n \in \N$ we define $L^{(\omega)}_n$ as the unique nonnegative integer such that
\begin{equation}\label{eq:defln-conv}
2^{-L^{(\omega)}_n} \leq a_{\omega_1} \cdots   a_{\omega_{rn}} < 2^{1 -L^{(\omega)}_n},
\end{equation}
and, as in the proof of Theorem \ref{thm:main-projections}, consider the nested families of intervals $\mathfrak{D}_n^{(\omega)}$ given by
\[
\mathfrak{D}^{(\omega)}_n=\{ [2^{-L^{(\omega)}_n}j,2^{-L^{(\omega)}_n}(j+1)) : j \in \mathbb{Z} \}.
\]
With this, for $q > 1$ we define the functions $\tau_{q,n}: Y \times S^1_\beta \rightarrow \R$
\begin{equation}\label{eq:taudefi-conv}
\tau_{q,n}(\omega,s) : = \sum_{I \in \mathfrak{D}^{(\omega)}_{n}} \left(\eta^{(\omega)}_s (I)\right)^q.
\end{equation}

Similarly to the proof of Theorem \ref{thm:main-projections}, we will show that $\tau_{q,n}$ is a submultiplicative cocycle (up to a multiplicative constant), and then we will construct a ``nicer'' cocylce $\overline{\tau}_{q,n}$ which is comparable to $\tau_{q,n}$. Unlike the situation in Theorem \ref{thm:main-projections}, the functions $\overline{\tau}_{q,n}$ will \emph{not} be continuous, but will nevertheless be approximable by continuous functions in a suitable way. Since $q$ will remain fixed, we drop it from the notation.

Hence, the first step is to show that there exists $K_1 > 1$ such that for any $n,m \in \N$ and $(\omega,s) \in Y \times S^1_\beta$, one has
\begin{equation}\label{eq:mcocycle-conv}
\tau_{n+m}(\omega,s) \leq K_1 \,\tau_n(\omega,s) \,\tau_{m}(\mathbf{S}^n(\omega,s)).
\end{equation}
To see this, let us fix $n,m \in \N$, $(\omega,s) \in Y \times S^1_\beta$, and proceed to show \eqref{eq:mcocycle-conv}. We will consider three separate cases, depending on whether $-\beta \leq \mathbf{R}^n(\omega,0) + s < \beta$, $\mathbf{R}^n(\omega,0) + s \geq \beta$ or $\mathbf{R}^n(\omega,0) + s < -\beta$. In the first case one has $\mathbf{R}^n(\omega,s)=\mathbf{R}^n(\omega,0) + s$ whereas in the second one has $\mathbf{R}^n(\omega,s)=\mathbf{R}^n(\omega,0)+s - 2\beta$ and in the third $\mathbf{R}^n(\omega,s)=\mathbf{R}^n(\omega,0)+s + 2\beta$ holds instead. For the proof of the first case we will only use the family $\W^{(\omega)}$ and replace it with the family $\mathbb{Y}^{(\omega)}$ for the proof of the second case and with $\mathbb{Z}^{(\omega)}$ for the proof of the third case. Except for this difference, the proof of all three cases is completely analogous so we will only comment on the first case only.

The proof is a minor variant of the proof of \eqref{eq:mcocycle}. Given $J \in \mathfrak{D}^{(\omega)}_n$ let us define
\[
\W_{J}^{(\omega)}(s):=\{ (u,v) \in \W^{(\omega)}_n : ([u]_\om  \times [v]) \cap \opi_s^{-1}(J) \neq \emptyset \}
\]
and consider the interval $\widehat{J}$ which has the same center as $J$ but of length $|\widehat{J}|= 65 e^{2\beta}|J|$ instead. The constant is chosen to ensure that $\opi_s([u]_\om\times [v]) \subset \widehat{J}$ for every $(u,v) \in \W_{J}^{(\omega)}(s)$.

If $I \in \mathfrak{D}_{n+m}^{(\omega)}$ is such that $I \subset J$, then
\[
 \eta^{(\omega)}_{s}(I) \leq K''  \sum_{(u,v) \in \W_{J}^{(\omega)}(s)} \oeta^{(\omega)}([u]_\om\times [v]) \; \eta^{(\mathbf{T}^{rn}\om)}(H^\pom_{u,v}\Pi_s^{-1}(I))
\]
This can be established in a very similar manner to \eqref{eq1}; we omit the details. If we continue to argue as in the proof of Theorem \ref{thm:main-projections}, we further obtain

\begin{align}
\sum_{I \in \mathfrak{D}^{(\omega)}_{n+m} \atop I \subset J}  \left(\eta^{(\omega)}_s(I)\right)^q &\leq (K'')^q \left(\eta^{(\omega)}_s(\widehat{J})\right)^{q-1}  \times\nonumber\\
&\times \sum_{(u,v) \in \W^{(\omega)}_{J}(s)} \eta^{(\omega)}(Q^{(\omega)}_{u,v})
\sum_{I \in \mathfrak{D}^{(\omega)}_{n+m} \atop I \subset J}
 \left(\eta^{(\mathbf{T}^{rn}(\omega))} \big( H^{(\omega)}_{u,v}\Pi_s^{-1} (I)\big)\right)^q,\label{eq:submult-interm-step-conv}
\end{align}
recall \eqref{eq:submult-interm-step} and \eqref{eq:exchanged-sum}. Now, using \eqref{eq:h} it is not hard to see that for any such interval $I$ and $(u,v) \in \W^{(\omega)}_J$ one has
\begin{equation}\label{eq:expression-inverse-conv}
H^{(\omega)}_{u,v}\Pi_s^{-1} (I) = \left(\Pi_s \circ h^{(\omega)}_{(u,v)}\right)^{-1}(I) = \Pi_t^{-1} \left(\frac{1}{a_{\omega_1} \cdots a_{\omega_{rn}}} \cdot (I - \Pi_s(t^{(\omega)}_u, t'_v))\right),
\end{equation}
where $t:= \mathbf{R}^n(\omega,s)$ (in fact, we get \eqref{eq:expression-inverse-conv} for $t=\mathbf{R}^n(\omega,0)+s$ which, in this case, coincides with $\mathbf{R}^n(\omega,s)$; this is the point where it is important to use the appropriate family of rectangles). Furthermore, the families $\mathfrak{D}^{(\omega)}_{n+m}$ and
\[
\left\{ \frac{1}{a_{\omega_1} \cdots a_{\omega_{rn}}} (I - \Pi_s( d^{(\omega)}_u,d_v)) : I \in \mathfrak{D}^{(\omega)}_{n+m}\right\}
\]
can be seen to be  $6$-equivalent, so that by Lemma \ref{lem:equivalent-families} and \eqref{eq:expression-inverse-conv} we obtain
\begin{equation}\label{eq:sumareversa-conv}
\sum_{I \in \mathfrak{D}^{(\omega)}_{n+m} \atop I \subset J} \left(\eta^{(\mathbf{T}^{rn}(\omega))} \big( H^{(\omega)}_{u,v} \Pi_s^{-1} (I) \big)\right)^q \leq 6^{q-1} \tau_m(\mathbf{T}^{rn}(\omega),t) = 6^{q-1} \tau_m( \mathbf{S}^n(\omega,s)).
\end{equation}
Combining this with \eqref{eq:submult-interm-step-conv}, and reasoning exactly as in the end of Section  \ref{subsec:mult-cocycle}, we finally deduce that the cocycle relation \eqref{eq:mcocycle-conv} holds for some $K_1>0$ depending on $q$.

\subsection{An upper $C$-approximable analogue of $\tau_{n}$} In order to apply Corollary \ref{cor:uniform-in-fiber}, we need a $C$-approximable cocycle (recall Definition \ref{def:C-approx}). Unlike the situation in Theorem \ref{thm:main-projections}, there is now an inherent discontinuity at the end point of the interval $[-\beta,\beta)$; note that although the identification of the extreme points is required for applying ergodic-theoretic tools, as far as the geometric definition of $\tau_n$ is concerned, there is no such identification. This issue arises already in \cite[p. 107]{NPS12}, where (in the course of proving what effectively is a special case of Theorem \ref{thm:mainconvolutions}) it is incorrectly claimed that the functions $\phi_n$ (corresponding to our $\tau_n$) are continuous. In fact, there is continuity up to the endpoint of the interval. Fortunately, this turns out to be a minor issue: because the discontinuity set is small, the new cocycle is still upper $C$-approximable.

We proceed to the details. Firstly, in close analogy to Section \ref{subsec:continuous-analog}, we define
\[
\psi_n^{(\omega)}(x,y) = \psi ( 2^{L_n^{(\omega)}}(x - y) ).
\]
where  $\psi \in C^\infty_0(\R)$ is supported on $[-2,2)$, $0 \leq \psi \leq 1$ and $\psi|_{[-1,1)} \equiv 1$, and
\[
\overline{\tau}_n(\omega,s):= \int_\R \left( \int_\R \psi_n^{(\omega)}(x,y) d\eta_s^{(\omega)}(x) \right)^{q-1} d\eta_s^{(\omega)}(y).
\]
Then one can check, just as in Section \ref{subsec:continuous-analog}, that there is a constant $M\ge 1$ such that
\begin{equation} \label{eq:cocycles-equivalent-conv}
M^{-1} \tau_n(\om,s)\le \overline{\tau}_n(\om,s) \le M\tau_n(\om,s) \quad\text{for all }n\in\N, (\om,s)\in Y\times S^1_\beta,
\end{equation}
and $\overline{\tau}_n$ is continuous on $Y\times (-\beta,\beta)$. Since we clearly have $0 \leq \overline{\tau}_n \leq 1$, the fact that $\overline{\tau}_n$ is $C$-approximable is now a consequence of the following lemma.

\begin{lemma}\label{lem:C-approx} Given a measure $\mu$ on $Y$, every bounded $f: Y \times S^1_\beta \rightarrow \R$ which is continuous on $Y \times (-\beta,\beta)$ is also upper $C$-approximable on $(Y\times S^1_\beta,\mu\times\mathcal{L}_\beta)$.
\end{lemma}

\begin{proof}
For $\delta \in (0,\beta)$ let $g:S^1_\beta \rightarrow [0,1]$ be a continuous function such that $g|_{[\beta-\frac{1}{2}\delta,-\beta+\frac{1}{2}\delta]} \equiv 0$ and $g|_{[-\beta+\delta,\beta-\delta]} \equiv 1$, where we identify $-\beta=\beta$.

The function $M_\delta : Y \times S^1_\beta \rightarrow \R$ defined by the formula
$$
M_\delta(\omega,s) = f(\omega,s)g(s) + \|f\|_\infty (1-g(s))
$$
is continuous on $Y \times S^1_\beta$ and also satisfies $f \leq M_\delta \leq 2 \|f\|_\infty$. Furthermore, since $M_\delta$ and $f$ agree on $Y\times [-\beta+\delta,\beta-\delta]$, we have that
$$
\int_{Y \times S^1_\beta}  (M_\delta - f) d( \mu \times \mathcal{L}_{\beta}) \leq 2\|f\|_\infty \mathcal{L}_\beta( [\beta-\delta,-\beta+\delta] )=  \frac{ 2\delta\|f\|_\infty}{\beta}
$$
which shows that $f$ is upper $C$-approximable.
\end{proof}

From this, it follows that if for each $n \in \N$ we define $\phi_n : Y \times S^1_\beta \rightarrow \R$ by the formula
\[
\phi_n := \log ( K_1 \overline{\tau}_n )
\]
where $K_1$ is as in \eqref{eq:mcocycle-conv}, then the sequence $(\phi_n)_{n \in \N}$ is a subadditive cocycle on $Y \times S^1_\beta$. Notice that each $\phi_n$ is well defined since $\tau_n$, and hence also $\overline{\tau}_n$, are strictly positive. Furthermore, by Lemma \ref{lem:C-approx} we get that $\phi_n$ will also be upper $C$-approximable provided that it remains bounded. This fact will be a consequence of \eqref{eq:cocycles-equivalent-conv} and the following lemma.

\begin{lemma}\label{lema1} For each $n \in \N$ we have
$$
0 < \inf_{(\omega,s) \in Y \times S^1_\beta} \tau_n(\omega,s) \leq \sup_{(\omega,s) \in Y \times S^1_\beta} \tau_n(\omega,s) < +\infty.
$$
\end{lemma}
\begin{proof} Notice that, since $\eta_s^{(\omega)}$ is a probability measure for each $(\omega,s)$, we have
\[
\tau_n(\omega,s) = \sum_{I \in \mathfrak{D}^{(\omega)}_n} \left(\eta^{(\omega)}_s(I)\right)^q \leq \sum_{I \in \mathfrak{D}^{(\omega)}_n} \eta^{(\omega)}_s(I) = \eta_s^{(\omega)}(\R)=1.
\]
To establish the other inequality we notice that for any $(\omega,s) \in Y \times S^1_\beta$ we have
\[
\text{Supp}(\eta_s^{(\omega)}) \subseteq \Pi_s( [0,1] \times [0,1] ) \subseteq [0,1+e^{\beta}]
\]
so that for each $n \in \N$ there exist at most $c_{n,\beta}$ intervals $I \in \mathfrak{D}^{(\omega)}_n$ satisfying $\eta_s^{(\omega)}(I) \neq 0$, where $c_{n,\beta} \in \N$ is a constant which, since $\inf_{i=1,\dots,N} a_i > 0$, can be chosen independently of $\omega$. Thus, for each $(\omega,s)$ there exists at least one $I \in \mathfrak{D}^{(\omega)}_n$ such that $\eta_s^{(\omega)}(I) \geq \frac{1}{c_{n,\beta}}$, which implies that
\[
\frac{1}{c^q_{n,\beta}} \leq \inf_{(\omega,s) \in \{1,\dots,N\}^\N \times S^1_\beta} \tau_n(\omega,s).
\]
\end{proof}

\subsection{Conclusion of the proof}

The remaining of the proof of Theorem \ref{thm:mainconvolutions} follows exactly the same lines as the proof of Theorem \ref{thm:main-projections}. In particular, \eqref{eq:cdimensionq} holds in the current setting. Details are left to the interested reader.

\section{Examples and applications}
\label{sec:examples}

\subsection{The deterministic case}

When there is just $N=1$ rule, we obtain the following immediate consequence of Theorem \ref{thm:main-projections}

\begin{corollary} \label{cor:projections-deterministic}
Let $\{f_j(x)= \lambda R_{\alpha} x + t_j\}_{j=1}^k$, where $\lambda\in (0,1)$, $R_\alpha$ is rotation by $\alpha\in [0,2\pi)$ and $t_j\in\R^2$ are translations. Let $\oeta$ be a measure on $\X_\infty:=\{1,\ldots,k\}^\N$ such that
\[
\oeta([uv]) \le K\, \oeta[u]\,\oeta[v]
\]
for some $K>1$ and all finite words $u\in \X_m, v\in\X_n$, and let $\eta$ be the projection of $\oeta$ under the coding map.

If $\alpha/\pi$ is irrational, then for all $q\in (1,2]$ and all $v\in S^1$ it holds that
\[
D_q(\Pi_v\eta) = \min(D_q\eta,1),
\]
and moreover the convergence of $-\frac{\log\mathcal{C}^q_{\Pi_v\eta}(n)}{n(q-1)}$ to $\min(D_q\eta,1)$ is uniform over $v\in S^1$.
\end{corollary}
\begin{proof}
The dynamical system $(Y\times S^1,\mathbf{S})$ degenerates to rotation by $\alpha$ on the circle, for which Lebesgue measure is certainly ergodic. This is then just a special case of Theorem \ref{thm:main-projections}.
\end{proof}

Measures $\oeta$ satisfying the assumptions include product (Bernoulli) measures on $\X_\infty$, as well as Gibbs measures for H\"{o}lder potentials and for almost-additive sequences of potentials. When $\oeta$ is Bernoulli, then $\eta$ is a self-similar measure on the corresponding self-similar set, so in particular we obtain existence and preservation of $L^q$ dimensions of projections of self-similar measures for homogeneous planar iterated function systems (for which the linear part contains an irrational rotation), regardless of overlaps. For Hausdorff dimension this is known to hold also for non-homogeneous systems \cite{HochmanShmerkin12, FalconerJin14}; below we will recover this as another consequence of Theorem \ref{thm:main-projections}.

In a similar way, we have the following consequence of (the proof of) Theorem \ref{thm:mainconvolutions}.

\begin{corollary} \label{cor:measure-on-product-deterministic}
For $(i,j)\in \{1,\ldots,k_1\}\times \{1,\ldots,k_2\}$, let
\[
f_{i,j}(x,y) = (ax + t_i, bx + u_j),
\]
where $0<a<b<1$ and $t_i,u_j\in\R$.

For $i=1,2$, let $\overline{\nu}_i$ be a measure on $\{1,\ldots,k_i\}^\N$ such that
\begin{equation} \label{eq:submult-determinist-product}
\overline{\nu}_i([uv]) \le K\, \overline{\nu}_i([u]) \, \overline{\nu}_i([v]),
\end{equation}
for any words $u\in \{1,\ldots,k_i\}^m, v\in \{1,\ldots,k_i\}^n$. Let $\nu_i$ be the projection of $\overline{\nu}_i$ under the respective coding map.

If $\log a/\log b$ is irrational, then for all $t>0$,
\[
D_q(\nu_1 * A_t\nu_2) = \min(D_q(\nu_1)+D_q(\nu_2),1),
\]
where $A_t x=tx$, and moreover
\[
-\frac{\log\mathcal{C}^q_{\nu_1 * A_t\nu_2}(n)}{(q-1)n} \to \min(D_q(\nu_1)+D_q(\nu_2),1)
\]
uniformly over compact subsets of  $(0,+\infty)$.
\end{corollary}
\begin{proof}
Since there is no code space, the ergodicity assumption in Theorem \ref{thm:mainconvolutions} reduces to $\mathcal{L}_b$ being ergodic for the map $s\mapsto s+_{\beta} \ln(b/a)$ on $S^1_\beta$, where $\beta=\ln(b^{-\ell})$ and $\ell\in\N$ is arbitrarily large (recall \eqref{control1} and \eqref{control}; we take $r=1$ since $a<b$). Since $\log b/\log a$ is irrational, these systems are isomorphic to irrational rotations for any value of $\ell$, so the claim follows from (the proof of) Theorem \ref{thm:mainconvolutions}.
\end{proof}

This extends \cite[Theorem 1.1]{NPS12}, and most of the generalizations outlined in \cite[Section 5]{NPS12}. More precisely, we allow overlapping in the construction, our measures on the Cantor sets are more general (including Gibbs measures), and we obtain uniform convergence over compact sets of scalings $t$.

\subsection{Random self-similar measures}

Next, we go back to the setting of Theorem \ref{thm:main-projections} with $N$ different rules, but assume that the measures $\oeta^\pom$ have the following product structure. For each $i\in\{1,\ldots,N\}$, let $p_i=(p_{i,1},\ldots,p_{i,k_i})$ be a probability vector, and set
\begin{equation} \label{eq:eta-product}
 \oeta^\pom =  \prod_{n=1}^\infty p_{\omega_i}.
\end{equation}
It is immediate that properties (a)-(c) hold. (Recall Section \ref{subsec:themodel}). Let $\mu$ be an ergodic measure for $(Y,\bt)$, where as usual $Y=\{1,\ldots,N\}^\N$.

We want to obtain an explicit formula for the $L^q$ dimensions of the projections $\eta^\pom$; for this, we need to assume some separation assumption. For simplicity we assume the following very strong separation condition:
\begin{equation} \label{eq:strong-separaton}
\text{for each }i\in\{1,\ldots, N\}\text{, the disks } f_1^{(i)}(B),\ldots, f_{k_i}^{(i)}(B) \text{ are disjoint},
\end{equation}
where $B=B[0,R]$ is a ball that such that $f_j^{(i)}B\subset B$ for all $i,j$, recall Section \ref{subsec:themodel}. This could be relaxed to a random open set condition, but we do not pursue this. The following lemma is standard, but we include the proof for completeness.

\begin{lemma} \label{lem:value-Lqdim}
Let the family $\oeta^\pom$ be given by \eqref{eq:eta-product}, and suppose \eqref{eq:strong-separaton} holds. Then for every $q>1$, the $L^q$ dimension of $\eta^{(\omega)}=\Delta_\om\oeta^\pom$ exists and is constant on the set of $\mu$-generic points $\om$, and is given by the value
\begin{equation}\label{formuladimension}
D^q(\eta^\pom) = \dfrac{\int \log{(p_{\widetilde{\omega}_1,1}^q + \cdots + p_{\widetilde{\omega}_1, k_{\widetilde{\omega}_1}}^q)} \,d\mu(\widetilde{\omega})}{(q-1) \int \log (\lambda_{\widetilde{\omega}_1}) \,d\mu (\widetilde{\omega})}.
\end{equation}
\end{lemma}

\begin{proof} For each $n \in \mathbb{N}$ let us define
$$
\theta_n(\omega) : = \sum_{u \in \X_n^\pom} \big( \eta^{(\omega)}( B^{(\omega)}_{u}) \big)^q.
$$
Note that we have $$D^q(\eta^{(\omega)}) = \lim_{n\to \infty} \frac{\log{\theta_{n}(\omega)}}{(q-1)\log(\lambda_{\omega_1} \cdots \lambda_{\omega_n})}.$$
Indeed, this follows from the definition of $L_n^{(\omega)}$ in Equation \eqref{eq:defln}, and the fact that the families $\mathfrak{D}_n^{(\omega)}$ and $\{B^{(\omega)}_{u} : |u|=n\}$ are $C$-equivalent on $\supp(\eta^\pom)$ for some $C>0$, by reasoning as in the proof of Equation \eqref{correlacion}.

Observe that in the $n$-th step of the construction of $\mathcal{C}^{(\omega)}$ we have $k_{\omega_1} \cdots k_{\omega_n}$ disks, each of measure $p_{\omega_1 , i_1} \cdots p_{\omega_n , i_n }$ (for a given choice of $i_j \in \{1, \cdots, k_{\omega_j} \}$). Then
\begin{align*}
\theta_n(\omega) &= \sum_{i_1 \in \{1, \cdots , k_{\omega_1}\}} \cdots \sum_{i_n \in \{1, \cdots , k_{\omega_n}\}} p_{\omega_1 , i_1}^q \cdots \;p_{\omega_n , i_n }^q \\
&= \prod_{j=1}^n (p_{\omega_j , 1}^q + \cdots + p_{\omega_j , k_{\omega_j}}^q ).
\end{align*}
Let $H(\omega):= \log{(p_{\omega_1,1}^q + \cdots + p_{\omega_1, k_{\omega_1}}^q )}$ (which is obviously continuous, since it depends only on $\omega_1$). Then notice that $\log{\theta_{n}(\omega)} = \sum_{j=1}^n H(\mathbf{T}^{j-1} (\omega))$.
Let $\mathcal{G}$ denote the set of $\mu$-generic points. If $\omega \in \mathcal{G}$, then
\begin{equation}\label{tau}
\frac{1}{n} \log{\theta_{n}(\omega)}   \longrightarrow \int H(\widetilde{\omega}) \,d\mu(\widetilde{\omega}) =  \int \log{(p_{\widetilde{\omega}_1,1}^q + \cdots + p_{\widetilde{\omega}_1, k_{\widetilde{\omega}_1}}^q )} \,d\mu(\widetilde{\omega}).
\end{equation}
Similarly, it was shown in the proof of Lemma \ref{lem:limite} that for any $\omega\in\mathcal{G}$,
\begin{equation}\label{segundo}
\frac{\log(\lambda_{\omega_1} \cdots \lambda_{\omega_n})}{n}  \longrightarrow
\int \log (\lambda_{\widetilde{\omega}_1}) d\mu (\widetilde{\omega}) .
\end{equation}

Therefore, by equations \eqref{tau} and \eqref{segundo}, we conclude
\begin{align*}
(q-1)D^q(\eta^{(\omega)}) &= \lim_{n\to +\infty} \frac{\log{\theta_{n}(\omega)}}{\log(\lambda_{\omega_1} \cdots \lambda_{\omega_n})} = \lim_{n\to +\infty} \frac{\frac{1}{n}\log{\theta_{n}(\omega)}}{\frac{1}{n}\log(\lambda_{\omega_1} \cdots \lambda_{\omega_n})} \\
&= \frac{\int \log{(p_{\widetilde{\omega}_1,1}^q + \cdots + p_{\widetilde{\omega}_1, k_{\widetilde{\omega}_1}}^q )} \,d\mu(\widetilde{\omega})}{\int \log (\lambda_{\widetilde{\omega}_1}) \,d\mu (\widetilde{\omega})}.
\end{align*}
This ends the proof.
\end{proof}

By applying this to the case in which $\mu$ is a Bernoulli measure, we obtain the following consequence of Theorem \ref{thm:main-projections}.
\begin{corollary} \label{cor:projections-product}
Let the family $\oeta^\pom$ be given by \eqref{eq:eta-product}, and suppose \eqref{eq:strong-separaton} holds. Let $r=(r_1,\ldots,r_N)$ be a probability vector and let $\mu$ be the $r$-Bernoulli measure. Finally, assume $\alpha_i/\pi$ is irrational for some $i$ with $r_i>0$.

Then for each $q\in (1,2]$ and each $\mu$-generic $\om$
\[
D_q(\Pi_v\eta^\pom) = \min(D(q),1) \quad\text{for all }v\in S^1,
\]
where
\[
D(q)= \frac{\sum_{i=1}^N  r_i \log(p_{i,1}^q + \cdots + p_{i, k_i}^q)}{ (q-1)\sum_{i=1}^N r_i \log (\lambda_i) }.
\]
Furthermore, the convergence of $-\frac{\log\mathcal{C}_{\Pi_v\eta^\pom}^q(n)}{(q-1)n}$ to $\min(D(q),1)$ is uniform in $v\in S^1$.
\end{corollary}
\begin{proof}
Ergodicity of $\mu\times\mathcal{L}$ for the skew-product given in \eqref{eq:def-skew-product} is classical when $\mu$ is a Bernoulli measure (provided some $\alpha_i/\pi$ is irrational), see e.g. \cite[Corollary 4.5]{Parry97}.  The claim is then immediate from Theorem \ref{thm:main-projections} and Lemma \ref{lem:value-Lqdim}.
\end{proof}
An analogous result is available in the setting of Theorem \ref{thm:mainconvolutions} (with one of the sets deterministic). Since it is possible to construct explicit generic points for Bernoulli measures, the above corollary applies also to some deterministic constructions.

\subsection{Uniform box-counting estimates}

An upper bound on $\mathcal{C}^q_{\mu} (n)$ for $q>1$ yields (via H\"{o}lder's inequality) a lower bound on the number of cubes in $\mathfrak{D}_n$ hit by $\supp(\mu)$:
\begin{equation} \label{eq:Holder}
1 = \sum_{I\in\mathfrak{D}_n} \mu(I)  \le \#\{I\in\mathfrak{D}_n: I\cap \supp(\mu)\neq\varnothing\}^{1/q'} \mathcal{C}^q_{\mu} (n)^{1/q}.
\end{equation}
Together with our main results, this yields uniform lower box-counting bounds for the projections of the supports of the measures in question. We give one concrete example.
\begin{corollary}
Let $A\subset\R^2$ be a self-similar set, that is, $A=\bigcup_{i=1}^k f_i(A)$ for some contracting similarities $f_i$. If the orthogonal part of some $f_i$ is an irrational rotation, then for every $\varepsilon>0$ there is $\delta>0$ such that for all $n\in\N$ and $v\in S^1$, the projection $\Pi_v A$ hits at least $\delta\,2^{(\gamma-\varepsilon)n}$ intervals in $\mathfrak{D}_n$, where $\gamma=\min(\dim_H(A),1)$.
\end{corollary}
\begin{proof}
We may assume that $f_i(x)=\lambda R_\alpha(x)+t_i$ where $\alpha/\pi$ is irrational and, moreover, the strong separation condition holds. Indeed, any planar self-similar set $A$ for which one of the generating maps contains an irrational rotation, contains self-similar sets of this special form and dimension arbitrarily close to that of $A$, see e.g. \cite[Lemma 4.2]{Shmerkin15}.

If $\oeta$ is the $(\tfrac 1k,\ldots,\tfrac 1k)$-Bernoulli measure on $\{1,\ldots,k\}^\N$, and $\eta$ is its projection onto $A$ via the coding map, then it is well known  that $D_2(\eta)=\dim_H(A)=\log k/|\log\lambda|$. The claim now follows from Corollary \ref{cor:projections-deterministic} and \eqref{eq:Holder}.
\end{proof}

\subsection{Projections of non-homogeneous self-similar measures}
\label{subsec:hochman}

We use an integral representation of self-similar measures to recover a result from \cite{HochmanShmerkin12} on the Hausdorff dimension of projections planar self-similar measures. Let $f_j(x) = \lambda_j R_{\alpha_j}(x)+t_j$, $j=1,\ldots,k$ be contractive similarities (i.e. $\lambda_j\in (0,1)$). Let $\overline{p}=(\overline{p}_1,\ldots,\overline{p}_k)$ be a probability vector and let $\nu$ the corresponding self-similar measure. That is, $\nu$ is the projection of the $\overline{p}$-Bernoulli measure $\overline{\nu}$ under the coding map $\Delta$ given by
\[
\{ \Delta(u) \} = \bigcap_{n=1}^\infty f_{u_1}\cdots f_{u_n}(B),
\]
where $B$ is a large enough ball that $f_j(B)\subset B$ for all $j$.

Fix a large integer $\ell$. For each $u\in \{1,\ldots,k\}^\ell$, let $N_j(u)$ count the number of times the symbol $j$ appears in $u$, and write $N(u)=(N_1(u),\ldots,N_k(u))$. Note that $N$ takes values in
\[
\Sigma:= \left\{ (\ell_1,\ldots,\ell_k): \ell_i\ge 0, \sum_i \ell_i=\ell\right\} \subset \{0,1,\ldots,\ell\}^k,
\]
so in particular $\#\Sigma\le(\ell+1)^k$  (this is a rough estimate, but sufficient for us; the key is that it has polynomial size in $\ell$).

Also, if $N(u)=N(v)$, then the maps $f_u$ and $f_v$ have the same linear part (and possibly different translation parts), where as usual $f_u=f_{u_1}\circ\cdots\circ f_{u_\ell}$. Hence, for each $\sigma\in\Sigma$, $\{ f_u: N(u)=\sigma\}$ is a valid rule in the sense of Section \ref{subsec:themodel}.

Our goal is to disintegrate $\overline{\nu}$ over the fibers of the map
\[
(u_n)_{n\in\N}\mapsto (N(u_{(j-1)\ell+1} \ldots u_{j\ell}))_{j\in\N}
\]
that splits $u$ into blocks of length $\ell$ and applies $N$ to each block. Although such disintegration exists in a very general setting (see e.g. \cite[Chapter 5]{EinWar11}), in this simple setting there is an explicit expression, which we now describe.

For each $\sigma\in\Sigma$, write
\[
r_\sigma= \sum_{u\in\{1,\ldots,k\}^\ell:N(u)=\sigma} \overline{p}_{u_1}\cdots \overline{p}_{u_\ell} =: \sum_{u\in\{1,\ldots,k\}^\ell:N(u)=\sigma} \overline{p}_u.
\]
Consider the conditional probability $p_\sigma$ on the fiber $\{ u\in\{1,\ldots,k\}^\ell: N(u)=\sigma\}$, extended to all of $\{1,\ldots,k\}^\ell$ by assigning zero mass to the complement of the fiber. Formally, $p_{\sigma,u} = \overline{p}_u/r_\sigma$ if $N(u)=\sigma$, and $p_{\sigma,u}=0$ otherwise.

Note that one can sample a sequence $u=(u_1,\ldots,u_\ell)$ according to $\nu$ in the following way: choose $\sigma\in\Sigma$ according to the probability vector $r$; then choose $u$ according to the probability vector $p_\sigma$. Thanks to the product structure of $\nu$, this extends to infinite sequences as follows.

For each $\om\in\Sigma^\N$, let $\oeta^\pom$ be the product measure $\prod_{i=1}^\infty p_{\om_i}$ (this is a measure on $\left(\{1,\ldots,k\}^\ell\right)^\N$, which we can identify with $\{1,\ldots,k\}^\N$ in the canonical way). Explicitly,
\[
 \oeta^\pom([u_1,\ldots, u_{n\ell}]) = p_{\om_1,(u_1\ldots u_\ell)} \cdots p_{\om_j, (u_{(n-1)\ell+1},\ldots,u_{n\ell})}.
\]
Finally, write $\mu$ for the $r$-Bernoulli measure on $\Sigma^\N$.

\begin{lemma}
\[
\overline{\nu}(\cdot) = \int_{\Sigma^\N} \oeta^\pom(\cdot) \,d\mu(\om).
\]
\end{lemma}
\begin{proof}
It is enough to check that both measures agree on any cylinder $A=[i_1 \ldots i_{n \ell}] =:[v_1\ldots v_n]$ where $v_j=(i_{(j-1)\ell+1}\ldots i_{j\ell})$ (since any cylinder splits as a finite union of such cylinders).  But
\begin{align*}
  \int \oeta^\pom(A) \,d\mu(\om) &= \int_{\Sigma^\N} p_{\om_1,v_1}\cdots p_{\om_n,v_n}\, d\mu(\om)\\
  &= \int_{\om_i=N(v_i)} p_{\om_1,v_1}\cdots p_{\om_n,v_n}\, d\mu(\om)\\
  &= \prod_{j=1}^n r_{N(v_j)} \prod_{j=1}^n p_{N(v_j),v_j}\\
  &= \prod_{j=1}^n \overline{p}_{v_j} = \overline{\nu}(A).
\end{align*}
\end{proof}

We have defined things so that the coding maps $\Delta_\om$ agree with the original coding map $\Delta$ for $\nu$, after the usual identification of $\left(\{1,\ldots,k\}^\ell\right)^\N$ with $\{1,\ldots,k\}^\N$. Hence, it follows from the last lemma that also
\begin{equation} \label{eq:integral-representation}
\nu = \int \eta^\pom \,d\mu(\om).
\end{equation}
This is the disintegration we referred to above, and it is preserved under orthogonal projections.

Unfortunately, $L^q$ dimensions do not play nicely with integral representations, but Hausdorff dimension does. This allows us to recover, via a rather different proof which avoids the machinery of measure-valued processes, the following result which was first obtained in \cite{HochmanShmerkin12} (we note, however, that the methods of \cite{HochmanShmerkin12} extend to higher dimensions, while our approach breaks down in dimension $d\ge 3$).

\begin{corollary}
Suppose $f_i$, $\overline{p}$ and $\nu$ are as above, and assume further that the separation condition \eqref{eq:strong-separaton} holds. Then
\[
\dim_H(\Pi_v\nu) = \min(\dim_H\nu,1)\quad\text{for all }v\in S^1.
\]
\end{corollary}
\begin{proof}
It follows from the definition of Hausdorff dimension of a measure and the representation \eqref{eq:integral-representation} that if $\dim_H(\Pi_v\oeta^\pom)\ge s$ for $\mu$ almost all $\om$, then $\dim_H\Pi_v\nu\ge s$. We will show that the former holds for all $v\in S^1$, with a value of $s$ that can be made arbitrarily close to $\min(\dim_H\nu,1)$ by taking $\ell$ large enough.

First of all, the separation assumption implies that
\[
\dim_H\nu = \frac{\sum_{i=1}^k \overline{p}_i \log(\overline{p}_i)  }{\sum_{i=1}^k \overline{p}_i \log(\lambda_i)} = \frac{\sum_{u\in\{1,\ldots,k\}^\ell}\overline{p}_u \log(\overline{p}_u) }{\sum_{u\in\{1,\ldots,k\}^\ell} \overline{p}_u \log(\lambda_u)},
\]
where $\lambda_u=\lambda_{u_1}\cdots\lambda_{u_k}$, see e.g. \cite[Theorem 5.2.5]{Edgar98}.

On the other hand, we obtain from Corollary \ref{cor:projections-product} that for all $\mu$-generic $\om$, all $v\in S^1$, and all $q\in (1,2]$,
\[
\dim_H(\Pi_v\eta^\pom)  \ge D_q(\Pi_v\eta^\pom) =  \min\left(\frac{\sum_{\sigma\in\Sigma}  r_\sigma \log(\sum_{u\in\{1,\ldots,k\}^\ell} p_{\sigma,u}^q)}{ (q-1)\sum_{\sigma\in\Sigma} r_\sigma \log (\lambda_\sigma) },1\right),
\]
where  $\lambda_\sigma=\lambda_u$ for any $u$ such that $N(u)=\sigma$. (For the left-most inequality, recall that $L^q$ dimension, $q>1$, is always a lower bound for Hausdorff dimension.) Letting $q\to 1^+$, and recalling the definitions of $r_\sigma, p_{\sigma,u}$, we infer
\begin{align*}
\dim_H(\Pi_v\eta^\pom) &\ge  \min\left(\frac{\sum_{\sigma\in\Sigma}  r_\sigma \sum_{u\in\{1,\ldots,k\}^\ell} p_{\sigma_u}\log p_{\sigma_u}   }{\sum_{\sigma\in\Sigma} r_\sigma \log (\lambda_\sigma)},1\right) \\
&=  \min\left(\frac{\sum_{u\in\{1,\ldots,k\}^\ell}  \overline{p}_u \log(\overline{p}_u/r_{N(u)})}{\sum_{u\in\{1,\ldots,k\}^\ell} \overline{p}_u \log (\lambda_u)},1\right) \\
&\ge \min(\dim_H\nu,1) - \frac{\sum_{u\in\{1,\ldots,k\}^\ell} \overline{p}_u \log(r_{N(u)})}{\sum_{u\in\{1,\ldots,k\}^\ell}  \overline{p}_u \log (\lambda_u)}\\
&= \min(\dim_H\nu,1) - \frac{\sum_{\sigma\in\Sigma} r_\sigma \log(r_\sigma)}{\ell \sum_{i=1}^k  \overline{p}_i \log (\lambda_i)}\\
&\ge \min(\dim_H\nu,1) + \frac{\log\#\Sigma}{\ell \sum_{i=1}^k  \overline{p}_i \log (\lambda_i)}\\
&\ge \min(\dim_H\nu,1) + \frac{k\log(\ell+1)}{\ell \sum_{i=1}^k  \overline{p}_i \log (\lambda_i)}\\
&\longrightarrow \min(\dim_H\nu,1) \text{ as } \ell\to\infty,
\end{align*}
where in the fifth line we used that the entropy of a probability vector of length $M$ is bounded by $\log M$.  This completes the proof.
\end{proof}


\end{document}